\documentclass[11pt,letterpaper]{amsart}
\usepackage{amsmath,amstext,amsthm,amssymb,amsxtra,bbm}
\usepackage[top=1.5in, bottom=1.5in, left=1.25in, right=1.25in]	{geometry}
\usepackage[normalem]{ulem}
\usepackage{txfonts} 
\usepackage[T1]{fontenc}
\usepackage{lmodern}
\usepackage{tikz}\usetikzlibrary{arrows}

\usepackage{euler}   
\usepackage{enumerate}

\makeatletter
\DeclareRobustCommand\widecheck[1]{{\mathpalette\@widecheck{#1}}}
\def\@widecheck#1#2{%
    \setbox\z@\hbox{\m@th$#1#2$}%
    \setbox\tw@\hbox{\m@th$#1%
       \widehat{%
          \vrule\@width\z@\@height\ht\z@
          \vrule\@height\z@\@width\wd\z@}$}%
    \dp\tw@-\ht\z@
    \@tempdima\ht\z@ \advance\@tempdima2\ht\tw@ \divide\@tempdima\thr@@
    \setbox\tw@\hbox{%
       \raise\@tempdima\hbox{\scalebox{1}[-1]{\lower\@tempdima\box
\tw@}}}%
    {\ooalign{\box\tw@ \cr \box\z@}}}
\makeatother

\usepackage{mathtools}
\mathtoolsset{showonlyrefs,showmanualtags}

\usepackage{hyperref} 
\hypersetup{
    colorlinks=true,       
    linkcolor=blue,          
    citecolor=magenta,        
    filecolor=magenta,      
    urlcolor=cyan           
}

\usepackage[msc-links]{amsrefs}

\theoremstyle{plain} 
\newtheorem{lemma}[equation]{Lemma}

\newtheorem{theorem}[equation]{Theorem}
\newtheorem{corollary}[equation]{Corollary}

\theoremstyle{definition}
\newtheorem{definition}[equation]{Definition}

\theoremstyle{remark}
\newtheorem{remark}[equation]{Remark}

\numberwithin{equation}{section}

\title[Discrete-time Sparse Domination on Martingale Spaces]{Weighted Estimation by Discrete-time Sparse Domination on Martingale Spaces}

\author[W. Chen]{Wei Chen}
\address{School of Mathematics, Yangzhou University, Yangzhou 225002, China}
\email {weichen@yzu.edu.cn}

\author[C. Y. Zhang]{Chaoyue Zhang}
\address{School of Mathematics, Yangzhou University, Yangzhou 225002, China}
\email {cyzhang$\_$yzu@163.com}

\author[G. G. Zhang]{Gege Zhang}
\address{School of Mathematics, Yangzhou University, Yangzhou 225002, China}
\email {ggzhang$\_$yzu@163.com}

\thanks{The research of W. Chen is supported by the National Natural Science Foundation of China(12671163, 12271469).}

\begin{document}
\begin{abstract}Lacey used sparse domination to study the sharp weighted norm estimate of the maximal function of predictable multipliers in discrete time filtration spaces. Domelevo, Petermichl, and \v{S}kreb developed the self similarity argument known as sparse domination in an abstract martingale setting with a continuous time parameter. In our investigation, we establish sparse domination for discrete-time martingale transforms, introducing the novel concept of conditional sparsity as a core property of our approach. 
The conditional sparsity framework enables derivation of sharp weighted estimates and a mixed-norm estimate \( A_p^\alpha A_r^\beta \) that improves upon known sharp \( L^p \) bounds. Moreover, we develop dedicated sparse domination specifically for Doob's maximal operator, recovering the sharp bound as a direct application. 
Finally, we focus on the application of sparse theory to quantitative two-weight estimates.
\end{abstract}

	\maketitle
\tableofcontents

\section{Introduction and Main Results}
This paper studies sparse domination in martingale spaces and its applications. Martingale spaces and harmonic analysis are connected through dyadic analysis. Dyadic analysis has been widely used in harmonic analysis. Over the past thirty years, dyadic analysis has played a significant role in harmonic analysis and has given birth to the theory of sparse domination. This is fully reflected in the following breakthrough developments.

The Walsh model serves as the dyadic analogue of Fourier analysis. Thiele \cite{MR2692998} studied the boundedness of the bilinear Hilbert transform on the Walsh model, and \cite{MR1491450} proved the boundedness of this transform. Petermichl represented the Hilbert transform as an average of dyadic Haar shifts over random dyadic grids \cite{MR1756958}, represented the Riesz transforms as averages of corresponding dyadic operators \cite{MR2367098}, and constructed Bellman functions to prove the $A_2$ conjecture for these dyadic operators, thereby establishing the $A_2$ conjecture for the Hilbert transform and the Riesz transforms \cite{MR2354322, MR2367098}. 
The dyadic martingale transform is the dyadic analogue of Calder\'on-Zygmund singular integral operators. Wittwer \cite{MR1748283} constructed a Bellman function to prove the $A_2$ conjecture for the dyadic martingale transform. Building upon this foundation and utilizing extrapolation theory, Petermichl and Volberg \cite{MR1894362} proved that the Ahlfors-Beurling transform depends linearly on the $A_p$ weight constant ($p \geq 2$), thereby solving the regularity problem for solutions to the Beltrami equation. Beznosova \cite{MR2433959} proved the linear dependence on the $A_2$ weight constant for the dyadic paraproduct operator. In 2012, Hyt\"onen \cite{MR2912709} represented general Calder\'on-Zygmund singular integral operator $T$ in terms of general dyadic shift operators, completely proving the $A_2$ conjecture. Then Lerner \cite{MR3127380,MR3085756} provided a simplification of the proof of the $A_2$ conjecture using norm sparse domination. This proof relied on local median oscillation estimates in \cite{MR2721744}.  

The essential characteristic of a sparse family $\mathcal{S}$ is that although the sets in $\mathcal{S}$ may overlap, $\mathcal{S}$ possesses a family of pairwise disjoint core subsets $\{E_Q\}_{Q \in \mathcal{S}}$. Since then, sparse domination has attracted significant attention, largely owing to Lerner's seminal work \cite{MR3127380,MR3085756} on norm sparse domination.
When the kernel of a Calder\'on-Zygmund singular integral operator $T$ satisfies a log-Dini condition, \cite{MR3521084} and \cite{MR4007575} independently improved norm sparse domination to pointwise sparse domination  using local median oscillation estimates. In 2017, Lacey \cite{MR3625108} utilized a stopping-time argument in place of local median oscillation estimates to inductively prove the existence of pointwise sparse domination; in this work, the log-Dini condition on the kernel was weakened to a Dini condition. Then research on numerous operators has significantly advanced through the application of sparse domination techniques.

Weighted estimates for multilinear maximal operators were studied in \cite{MR3232584} using sparse domination, establishing their optimality. Assuming weights satisfy the reverse H\"older's inequality, the first author and Damian provided a multilinear version of Sawyer's conclusion in \cite{MR3118310} through sparse estimates of multilinear maximal operators. For Bochner-Riesz operators, bilinear sparse domination was established in \cite{MR3653057}, leading to new weighted estimates and vector-valued inequalities. Sparse domination for variational Carleson operators was achieved in \cite{MR3829751}, involving \( L_p \) averages and directly implying \( L_p \) boundedness along with corresponding weighted and vector-valued conclusions. A new maximal operator was introduced in \cite{MR4018107}, where weak-type inequalities for this operator enabled bilinear sparse domination for rough homogeneous singular integral operators. Sparse estimates for spherical maximal operators were obtained in \cite{MR4041115} using the \( L_p \) improving property of local spherical maximal operators, yielding weighted estimates involving \( A_p \) weights and the reverse H\"older's inequality. Most recently in 2024, paired sparse operators were utilized in \cite{MR4721778} to achieve pointwise dominations of multilinear operators and their commutators. For more information, see \cite{MR4772265}. 
 
Sparse domination essentially consists of two steps: first, dominating the target object using a sparse form; second, estimating the sparse form via relevant norms. This methodology advances developments in martingale theory.  

In discrete martingale spaces, Tanaka and Terasawa \cite{MR3004953} pioneered sparse domination for the Doob maximal operator and established quantitative norm estimates. Notably, in martingale spaces, the dyadic system characteristics become irrelevant. Their approach constructed a martingale analogue of sparse families---termed principal sets\footnote{ For more details on the principal cube construction, see \cite[p.131]{MR2657437}.}. Subsequently, \cite{MR4125846} demonstrated that these principal sets satisfy conditional sparsity properties (defined via conditional expectations), thereby deriving mixed weighted norm inequalities. For the multilinear case, by constructing a family of principal sets for the multilinear Doob's maximal operator, \cite{MR4244905} obtained results related to \( A_p \) weights and \( S_p \) weights. 

In the continuous index case, consider a filtered probability space $(\Omega, \mathcal{F}, \mu, \mathfrak{F})$, where  
\(\mathfrak{F} = (\mathcal{F}_t)_{t \geq 0}\)
is a right-continuous filtration with \(\mathcal{F}_0\) containing all \(\mu\)-null sets in \(\mathcal{F}\).  Let \( X \) and \( Y \) be uniformly integrable c\`{a}dl\`{a}g martingales. Domelevo and Petermichl \cite{MR3916937} proved that the differentially subordinate martingale (\cite{MR1370109,MR1334160}) linearly depends on the weight constant:  
\[
\| Y \|_{L^2(w)} \leq C[w]_{A_2}\| X \|_{L^2(w)}.
\]  
This proof method (via Bellman function) is ineffective for the maximal operator \( Y^* \) of the differentially subordinate martingale \( Y \). 
Assuming continuous paths, Ba\~nuelos, Brzozowski, and Os\c{e}kowski \cite{MR4143411} solved the corresponding maximal operator problem. In 2025, sparse domination for differentially subordinate martingales began to emerge (\cite{MR4926944}).
When \( Y \) is differentially subordinate to \( X \) , Domelevo, Petermichl, and \v{S}kreb \cite{MR4926944} provide a sparse domination and weighted estimate for
 the maximal function \( Y^* \) of $Y$ in greatest generality. This sparse domination is based on stochastic processes and stopping times, applies to c\`{a}dl\`{a}g martingales, and the provided estimates are dimension-free.  

Currently, sparse domination has greatly promoted the systematic study of various operators in harmonic analysis. However, research on sparse domination in martingale spaces remains comparatively underdeveloped. This paper investigates sparse domination in discrete martingale spaces and its applications.

Theorem \ref{thm:sparse} provides sparse domination for uniformly bounded martingale transforms on filtered probability spaces 
$(\Omega, \mathcal{F}, \mu, \{\mathcal{F}_n\}_{n \ge 0})$, extending Lacey's framework \cite{MR3625108} beyond atomic settings. 
In atomic filtrations--where $\sigma$-algebras $\mathcal{F}_n$ are generated by countable partitions and 
conditional expectations $\mathbb{E}[\cdot|\mathcal{F}_n]$ reduce to countably-valued functions---sparse 
collections consist of atoms satisfying measure conditions. For general martingale spaces with non-atomic 
elementary blocks, sparsity characterization requires conditional expectations, leading to our innovation: Conditional Sparsity. All  unexplained notations in martingale spaces can be found in Section \ref{sec:mar_theory}.

\begin{theorem}\label{thm:sparse}Let \( f \in L^1(\mu) \) and let \( Tf \) be its martingale transform corresponding to a multiplier sequence \( v = (v_n)_{n \geq 0} \) with \( \sup_n\|v_n\|_{\infty} \leq 1 \). There exists a constant \( C > 0 \) such that, for every such \( f \) and every such transform \( T \), there exists a sparse operator \( S = S_{T,f} \) satisfying
\begin{equation}\label{eq:main}
 |Tf| \leq C \cdot S|f|.
\end{equation}
The same inequality holds for \( M(Tf) \).
\end{theorem}

The application of conditional sparsity allows us to derive Theorem \ref{Theorem_SXC}. Combined with the pointwise domination in Theorem \ref{thm:sparse}, this yields Theorem \ref{thm:mar_trans}, which establishes the sharp weighted bound for uniformly bounded martingale transform---a result first obtained by Thiele, Treil, and Volberg \cite{MR3406523} in the context of atomic martingale spaces.

\begin{theorem}\label{Theorem_SXC}For $1<p<+\infty,$ we have 
\begin{equation} \label{thm:sharp}
\| S|f| \|_{L^p (w)} \lesssim
  [w]_{A_p} ^{\max \{1, \frac1{p - 1}\}} \| f \|_{L^p (w)} .
  \end{equation}
\end{theorem}

\begin{theorem}\label{thm:mar_trans}Let \( f \in L^1(\mu) \) and let \( Tf \) be its martingale transform corresponding to a multiplier sequence \( v = (v_n)_{n \geq 0} \) with \( \sup_n\|v_n\|_{\infty} \leq 1 \). Then we have
\begin{equation}\label{eq:mar_trans}
\|  Tf\|_{L^p (w)} \lesssim
  [w]_{A_p} ^{\max \{1, \frac1{p - 1}\}} \| f \|_{L^p (w)} .
\end{equation}
The same inequality holds for \( M(Tf) \).
\end{theorem}

As shown above, we directly employ sparse operators $S$ to derive sharp $A_p$ estimates for uniformly bounded martingale transforms. To push these estimates further, we introduce auxiliary sparse operators $S_{ \nu} f$, whose deployment within a framework of mixed weighted estimates leads to an even sharper bound. Let \( \nu \geq 1 .\)  The auxiliary operator $S_{ \nu}$ defined by
\[
S_{ \nu} f = \left( \sum_j \mathbb{E}\left( |f|\mid \mathcal{F}^{(i)}_{\tau^{(i)}}\right)^{ \nu} \mathbf{1}_{\{\tau^{(i)}<\infty\}} \right)^{1/ \nu},
\]
extends the sparse operator \( Sf \) (Definition \ref{def:sparse}). 
We establish a mixed-norm estimate of type $A_p^{\alpha}A_r^{\beta}$  in Theorem \ref{Theorem_SM}. Following Lerner's approach in \cite{MR3078357},  we derive Theorem \ref{Theorem_SM} via  standard duality arguments.

\begin{theorem}\label{Theorem_SM}  For $\nu+1\leq p\leq r<\infty,$ we have \[ \| S _{\nu}|f |\|_{L^p (w)} \lesssim
    [w]_{(A_p)^{\frac{1}{p-1}} (A_r)^{\frac{1}{\nu}-\frac{1}{p-1}}}  \| f \|_{L^p (w)} .\]
\end{theorem}

\begin{corollary}\label{corollary_imp} For $1< p\leq2\leq p^\prime\leq r<\infty,$ we have \[ \| S |f |\|_{L^p (w)} \lesssim
    [\sigma]_{(A_{p^\prime})^{\frac{1}{p^\prime-1}} (A_r)^{1-\frac{1}{p^\prime-1}}}  \| f \|_{L^p (w)} ,\]
    where $\sigma=w^{-\frac{1}{p-1}}.$
\end{corollary}

In the case of $1<p<2,$ the following norm comparison holds: $$ [\sigma]_{(A_{p^\prime})^{\frac{1}{p^\prime-1}} (A_r)^{1-\frac{1}{p^\prime-1}}}\leq[\sigma]_{(A_{p^\prime})^{\frac{1}{p^\prime-1}} (A_{p^\prime})^{1-\frac{1}{p^\prime-1}}}=[\sigma]_{A_{p^\prime}}=[w]^{\frac{1}{p-1}}_{A_{p}}.$$
This demonstrates that Corollary \ref{corollary_imp} provides a sharper bound than Theorem \ref{Theorem_SXC}.
 
Set \( v_n \equiv 1 \) for all \( n \geq 0 \) in Definition \ref{d:transform1}. For a martingale \( f = (f_n)_{n\geq 0} \) with \( f_0 = 0 \), we have \( M(Tf) = Mf \). In view of Theorem \ref{thm:mar_trans}, we obtain
\begin{equation}\label{eq:spec}
\|Mf\|_{L^p(w)} \lesssim [w]_{A_p}^{\max\left\{1,\frac{1}{p-1}\right\}} \|f\|_{L^p(w)}.
\end{equation}
Comparing with the sharp bound for Doob's maximal operator (\cite[Corollary4.5]{MR3004953}):
\begin{equation}\label{eq:gel}
\|Mf\|_{L^p(w)} \lesssim [w]_{A_p}^{\frac{1}{p-1}} \|f\|_{L^p(w)}, 
\end{equation}
we note that when \( p > 2 \), the exponent in \eqref{eq:spec} reduces to \( \frac{1}{p-1} \), matching \eqref{eq:gel}. However, for \( 1 < p < 2 \), 
\[
\max\left\{1, \frac{1}{p-1}\right\} = \frac{1}{p-1} > 1,
\]
while the optimal exponent in \eqref{eq:gel} remains \( \frac{1}{p-1} \). This indicates that for the special case \( v_n \equiv 1 \) (Doob's maximal operator), the sparse domination in Theorem \ref{Theorem_SXC} is not sharp, motivating the need for improved sparse constructions.  
In Theorem \ref{thm:max_sparse}, we establish a sparse domination specifically for Doob's maximal operator. As an application, this allows us to recover the sharp bound \eqref{eq:gel}.

\begin{theorem}\label{thm:max_sparse}
There is a constant \( C > 0 \) so that for all  \( f \in L^1(\mu) \), there is a simple sparse operator \( \mathbb{S} = \mathbb{S}_{M,f} \) such that
\begin{equation}\label{eq:max_main}
Mf \leq C \cdot \mathbb{S}|f|.
\end{equation}
\end{theorem}

\begin{theorem}\label{thm_Ap} Let $v$ be a weight  and $1< p< \infty.$ 
We have the inequality
\begin{equation}\label{norm}
\lVert M  f\rVert _{L ^{p} (w)} \leq C \lVert f\rVert _{L ^{p} (w)}
\end{equation}
if and only if $w\in  A_p$.
Moreover, if we denote the smallest constant in \eqref{norm} by $\|M\|$, we have
\begin{equation}\label{Ap_con}
[w]_{A_p}\leq\|M\|^p
\end{equation}
and
\begin{equation}\label{M_con}
\|M\|\lesssim[w]^{\frac{1}{p-1}}_{A_p}.
\end{equation}
\end{theorem}

Finally, leveraging sparse domination techniques,
we provide a unified perspective on two-weight estimates for Doob's maximal operator. The sparse domination of the Doob maximal operator established in Theorem \ref{thm:max_sparse} serves as the foundational result from which all three parts of Theorem \ref{thm-es} naturally follow. This approach not only yields weighted estimates but also reveals the underlying structural reasons for their validity. Our first result generalizes the $B_p$ framework of Hyt\"{o}nen and P\'erez \cite{MR3092729} to filtered measure spaces, eliminating the double supremum phenomenon present in \cite{MR3004953}. The second result embodies the core Hyt\"{o}nen-P\'erez philosophy---replacing partial $A_p$ control with weaker $A_{\infty}^*$ conditions---through a probabilistic sparse domination framework. The third estimate incorporates both the $A_{p^{\prime}}$ characteristic of $\sigma$ and its $A_{\infty}^*$ constant, where the logarithmic term captures the complexity of the stopping time selection process. While corresponding to Lerner-Moen type estimates \cite{MR3145553}, our sparse methodology in the filtered setting provides a probabilistic interpretation of the logarithmic factor, highlighting the versatility of sparse techniques across different mathematical contexts.

\begin{theorem} \label{thm-es}Let $1<p<\infty.$
\begin{enumerate}[\rm(1)]
                 \item \label{Bound1}If $(u,w)\in B_p,$ then $\|M\|_{L^p(w)\rightarrow L^p(u)}\lesssim[u,w]_{B_p}^{\frac{1}{p}};$
                 \item \label{Bound2}If $(u,w)\in A_p$ and $\sigma=w^{-\frac{1}{p-1}}\in A^*_{\infty},$ then $\|M\|_{L^p(w)\rightarrow L^p(u)}\lesssim[u,w]_{A_p}^{\frac{1}{p}}[\sigma]_{A^*_{\infty}}^{\frac{1}{p}};$
                 \item\label{Bound3}If $w\in A_p$ and $\sigma=w^{-\frac{1}{p-1}},$ then $\|M\|_{L^p(w)\rightarrow L^p(w)}\lesssim[\sigma]_{(A_{p'})^{\frac{1}{p'}}(A^*_{\infty})^{\frac{1}{p}}}\Big(1+\log_2[w]_{A_p}\Big)^{\frac{1}{p}}.$
\end{enumerate}
\end{theorem}

The article is organized as follows. In Section \ref{sec:mar_theory}, we present preliminary definitions and foundational results. Section \ref{sec:mar_trans} is devoted to proofs for the martingale transform, while Section \ref{sec:mar_max} contains proofs for Doob's maximal operator.

Throughout this work, the letters \( C \), \( C_1 \), and \( C_2 \) denote positive constants that may vary between occurrences. We employ the notation \( A \lesssim B \) to indicate the existence of an absolute constant \( C > 0 \), independent of the weight constant, such that \( A \leq CB \). Similarly, we write \( A \approx B \) when both \( A \lesssim B \) and \( B \lesssim A \) hold.

\section{Preliminaries}\label{sec:mar_theory}
This section establishes the theoretical foundation for sparse operator theory in martingale settings. It begins with a review of fundamental concepts—conditional expectation, martingales, martingale transforms, and stopping times—along with essential convergence theorems for martingale transforms. The new framework of locally filter-shifted stopping times and conditional sparsity is then introduced, forming the basis for the definitions of sparse operators and simple sparse operators. The section concludes with a review of weights, which are essential for the weighted norm inequalities developed later.

\subsection{Conditional Expectations and Martingales}
Let \((\Omega, \mathcal{F}, \mu)\) be a complete probability space, \(\mathcal{F}_1 \subset \mathcal{F}\) be a complete sub-$\sigma$-field. For \(f \in L^1(\Omega, \mathcal{F}, \mu)\),  let
\[
\eta(F) = \int_F f d\mu, \text{for all } F \in \mathcal{F}_1.
\] 
Then $\eta$ is a complex measure on $\mathcal{F}_1$ and absolutely continuous with respect to \(\mu|_{\mathcal{F}_1}\). In view of Radon-Nikod\'{y}m's theorem, there is a unique function (the Radon-Nikod\'{y}m derivative of 
$\eta$  with respect to $\mu$) denoted by \(h ,\) which is measurable with respect to \(\mathcal{F}_1\) and integrable with respect to \(\mu,\) and such that
\[
\int_F hd\mu = \int_F f d\mu, \text{for all } F \in \mathcal{F}_1.
\]
This bridges measure theory with probabilistic conditioning, providing the mathematical foundation for the rigorous definition of conditional expectation.

\begin{definition}
Let $(\Omega, \mathcal{F}, \mu)$ be a complete probability space, $\mathcal{F}_1 \subset \mathcal{F}$ a complete sub-$\sigma$-field, and $f \in L^1(\Omega, \mathcal{F}, \mu)$. The conditional expectation of $f$ with respect to $\mathcal{F}_1$, denoted by \(\mathbb{E}_{\mathcal{F}_1}(f)\),  \(\mathbb{E}(f|\mathcal{F}_1)\) or \(f_1\), is the unique $\mathcal{F}_1$-measurable function satisfying
\[
\int_F \mathbb{E}(f|\mathcal{F}_1)  d\mu = \int_F f  d\mu, \text{for all }  F \in \mathcal{F}_1.
\]
\end{definition}

Before recalling the definition of a martingale, we must first formally define the notion of a filtered probability space. 

\begin{definition}
A filtered probability space is a quadruple $(\Omega, \mathcal{F}, \mu, \{\mathcal{F}_n\}_{n \ge 0})$ where
\begin{itemize}
    \item[(P1)]  $(\Omega, \mathcal{F}, \mu)$ is a complete probability space.
    \item[(P2)]  $\{\mathcal{F}_n\}_{n \ge 0}$ is a filtration (an increasing family of sub-$\sigma$-algebras of $\mathcal{F}$).
    \item[(P3)]  Each $(\Omega, \mathcal{F}_n, \mu)$ is complete.
    \item[(P4)]  $\mathcal{F} = \bigvee_{n \geq 0} \mathcal{F}_n$ (the $\sigma$-algebra generated by the union).
\end{itemize}
\end{definition}

This structure provides the temporal framework necessary for defining martingales and other stochastic processes.

\begin{definition} \label{d:adapted} Let $ v= (v_n)_{n \ge 0}$ be a process on $(\Omega, \mathcal{F}, \mu, \{\mathcal{F}_n\}_{n \ge 0})$.Then $v$ is said to be adapted, if $v_n$ is $\mathcal{F}_n$ measurable, for all $n \ge 0$. Its maximal operator is defined as \( Mv = \sup |v_n|.\)
\end{definition}

The concept of adapted processes leads naturally to the fundamental class of martingales, which are characterized by their distinctive conditional expectation properties. Martingales play a central role in modern stochastic analysis due to their rich mathematical structure and wide applications.

\begin{definition}\label{d:adapted} Let $f = (f_n)_{n \ge 0}$ be an adapted process on $(\Omega, \mathcal{F}, \mu, \{\mathcal{F}_n\}_{n \ge 0})$. Then $f$ is said to be a martingale (with respect to $\{\mathcal{F}_n\}_{n \ge 0}$) if each $f_n \in L^1$, and
\[
\mathbb E(f_{n+1} | \mathcal{F}_n) = f_n, \quad n = 0, 1, 2, \cdots.
\]
\end{definition}

Having established the fundamental concept of martingales, we now turn to their important subclasses characterized by integrability conditions. The following definition introduces two particularly significant classes of martingales based on their $L^p$-boundedness properties.

\begin{definition}\label{}
Let \(1\leq p<\infty\). For any martingale \(f=(f_n)_{n\geq0}\), denote \(\|f\|_p = \sup_n \|f_n\|_p\). When \(\|f\|_p < \infty\), we say that \(f\) is an \(L^p\)-bounded martingale in symbols \(f\in L^p\). When \(f_n = \mathbb E(f_{\infty}|\mathcal{F}_n)\) for all \(n\geq0\), for some \(f_\infty\in L^p\), we say that \(f\) is an \(L^p\)-uniformly integral martingale  in symbols \(f\in L_u^p\).
\end{definition}

For an \(L^1\)-bounded martingale, its maximal operator $M$ satisfies the weak type $(1,1)$ bound which can be found in \cite[Theorem 2.1.1]{MR1224450}.  If $f\in L_u^1,$ we invoke the following Lemma \ref{lem: weak}. 

\begin{lemma}\cite[p.34]{MR1224450}\label{lem: weak}
The maximal operator \( M \) is an operator of weak type \( (1, 1) \), that is for all \( f \in L_u^1 \) and \(\lambda > 0\), we have
\[
\lambda\mu( Mf > \lambda )\leq  \| f_\infty \|_1.
\]
\end{lemma}

\begin{remark}\label{re: uniform}
It is known (\cite[p. 28]{MR1224450}) that when \(1 < p < \infty\), \(L^p = L^p_u\), and \(\|f\|_p=\|f_\infty\|_p\), but \(L^1_u\subsetneq L^1\) in general. For all \(f=(f_n)\in L^1\), we have that \(\lim_{n\rightarrow+\infty}f_n\) exists pointwise and \(\|\lim_{n\rightarrow+\infty}f_n\|_1\leq\|f\|_1\). Furthermore,  \(\|\lim_{n\rightarrow+\infty}f_n\|_1=\|f\|_1\) holds if and only if \(f=(f_n)\in L^1_u\), where \(\lim_{n\rightarrow+\infty}f_n=f_{\infty}\). When we  discuss \(L^p_u(p\geq1)\), we  may abuse notation by writing simply $f_\infty\in L^p$ instead of  $f=(f_n)\in L^p_u$ without  causing confusion. In the case, we often suppress the subscripts $\infty$ and $u.$ Thus, there is no confusion for the symbol \(L^p\) which denotes the usual Lebesgue spaces as well as the spaces of martingales. 
\end{remark}

\subsection{Martingale Transforms}
Lemma \ref{lem: weak} provides the essential foundation for studying the convergences of martingales. To systematically exploit the foundation,
Burkholder \cite{MR208647} used martingale transforms to study the interaction between martingales and adapted processes. Here and in what follows, we keep the convention: for any process $(\gamma_n)_{n \ge 0}$, $\gamma_{-1}$ is meant as $0$ except when otherwise stated. 

\begin{definition}\label{d:transform}Let $f  = (f_n)_{n \ge 0}$ be a martingale and $v = (v_n)_{n \ge 0}$ be an adapted process on $(\Omega, \mathcal{F}, \mu, \{\mathcal{F}_n\}_{n \ge 0})$.  Let $\Delta_0f = f_0$, $\Delta_1f = f_1 - f_0$, $\cdots$ so that $f_n=\sum_{k = 0}^{n}\Delta_kf$, $n\ge0$. Then the following transform
\[
g_n = \sum_{k=1}^n v_{k-1} \Delta_k f, \quad n \ge 1, \quad g_0 = 0,
\]
where $f = (f_n)_{n \ge 0}$ is a martingale, is called a martingale transform on $(\Omega, \mathcal{F}, \mu, \{\mathcal{F}_n\}_{n \ge 0})$. 
\end{definition}

\begin{remark}When $g_n \in L^1$, for all $n\geq0$, $g = (g_n)_{n \geq 0}$ is also a martingale provided $f$ is. This condition is satisfied, for example, if each $v_n$ is bounded. 
\end{remark}

The following fundamental result links martingale transforms to maximal function theory. It extends Doob's martingale convergence theorem to transformed processes, with convergence guaranteed on the set where the multiplier has finite maximal function.

\begin{lemma}\cite[Theorem 1]{MR208647}\label{lem:conver}
Let \( f = (f_n)_{n \ge 0}\) be an \( L^1 \)-bounded martingale and \( g \) be \( f \)'s martingale transform with the multiplier \( v = (v_n)_{n \ge 0} \).
Then \( g = (g_n)_{n \ge 0} \) converges a.e. on the set \( \{ Mv(x)  < \infty \} \).
\end{lemma}

Having established the basic martingale transform framework and its convergence properties, Burkholder \cite{MR542135} considered a practically significant specialization: transforms induced by uniformly bounded adapted processes. 

\begin{definition}\label{d:transform1}Let $v = (v_n)_{n \ge 0}$ be an adapted process  on $(\Omega, \mathcal{F}, \mu, \{\mathcal{F}_n\}_{n \ge 0}).$ If $v$ is uniformly bounded, then the following transform
\begin{equation}\label{d:mar-trans}
T f = \sum_{n=1}^{\infty} v_{n-1} \Delta_n f, \quad \text{for all }  L^1 \text{-bounded martingale } f = (f_n)_{n \ge 0}, 
\end{equation}
is also called a uniformly bounded martingale transform on $(\Omega, \mathcal{F}, \mu, \{\mathcal{F}_n\}_{n \ge 0})$. Moreover, let
\[
(Tf)_n= \sum_{k=1}^n v_{k-1} \Delta_k f, \quad n \ge 1, \quad (T f)_0 = 0,
\]
which, with no risk of
confusion, are denoted by 
\[Tf = ((T f)_n)_{n \ge 0}.\]
\end{definition}

Because of the convergence result (Lemma \ref{lem:conver}), $T f $ is well defined. Then the following weak-type $(1,1)$ inequality holds for uniformly bounded martingale transforms.

\begin{lemma}\cite[Theorem 1]{MR542135}\label{lem:weak_inequality} Let \( f = (f_n)_{n \ge 0}\) be an \( L^1 \)-bounded martingale and \(Tf \) be \( f \)'s martingale transform with the multiplier \( v = (v_n)_{n \ge 0} \). If  \( v \) is uniformly bounded in absolute value by 1. Then
\[
\lambda \mu(M(Tf) \ge \lambda) \le 2\sup_n \|f_n\|_1, \quad \lambda > 0.
\]
\end{lemma}

As an immediate consequence of  Lemma \ref{lem:weak_inequality}, we obtain the following distributional estimate.

\begin{corollary}\label{c:weak 1 1} Let \(f\in L^1\) and \(Tf \) be \( f \)'s martingale transform with the multiplier \( v = (v_n)_{n \ge 0} \). If  \( v \) is uniformly bounded in absolute value by $1,$ then
\[
\lambda \mu(Tf \ge \lambda)\leq\lambda \mu(M(Tf) \ge \lambda) \le 2 \|f\|_1, \quad \lambda > 0.
\]
\end{corollary}

\subsection{Stopping Times and Sparse Operators}
To further develop the localization techniques essential for martingale theory, we now present the cornerstone concept of stopping times, which play a pivotal role in optional sampling theorems and localization arguments and provide the foundation for formally defining sparse operators.

\begin{definition}\label{d:stopping} Denote the set of all nonnegative integers by $\mathbb{Z}^+$, and $\mathbb{Z}^+ \bigcup \{\infty\}$ by $\overline{\mathbb{Z}}^+$. A mapping $\tau$ from $\Omega$ to $\overline{\mathbb{Z}}^+$ is called a stopping time on $(\Omega, \mathcal{F}, \mu, \{\mathcal{F}_n\}_{n \ge 0})$, if $\{\omega : \tau(\omega) = n\} \in \mathcal{F}_n$, for all $n$, or equivalently $\{\omega : \tau(\omega) \leq n\} \in \mathcal{F}_n$, for all $n$.
\end{definition}

Building upon this, we invoke the associated $\sigma$-algebra at time $\tau$, which captures the information available up to this stopping time.

\begin{definition}\label{d:ahead stopping} 
Let $\tau$ be a stopping time on $(\Omega, \mathcal{F}, \mu, \{\mathcal{F}_n\}_{n \ge 0})$. The filtration at $\tau$, denoted $\mathcal{F}_\tau$, is the set of all events $F \in \mathcal{F}$ such that for every $n \geq 0$, the intersection $F \cap \{\tau \leq n\}$ belongs to $\mathcal{F}_n$. Formally:
\[
\mathcal{F}_\tau = \{ F \in \mathcal{F} : F \cap \{ \tau \leq n \} \in \mathcal{F}_n \text{ for all } n \}.
\]

\end{definition}

\begin{remark} \label{r:sigma} In the definition, $\{\tau \leq n\}$ could be replaced by $\{\tau = n\}$. And, $\mathcal{F}_T$ is a $\sigma$-field obviously. We think of $\mathcal{F}_n$ as the collection of all events observed up to time $n$, then $\mathcal{F}_T$ is the natural generalization of the $\sigma$-algebra 
$\mathcal{F}_n$ (information up to deterministic time $n$)
 to a random time $\tau$ and represents the information known (or events decidable) at the random time $\tau$. Notice that according to the definition, all subsets of $\{T = \infty\}$ belong to $\mathcal{F}_T$.
\end{remark}

We recall an elementary property of stopping times. In discrete time, the pointwise value \( f_\tau \) of a uniformly integrable martingale at a stopping time \( \tau \) is identically equal to the conditional expectation of $f$ with respect to the filtration up to \( \tau \). This result is fundamental in optimal stopping theory and stochastic control, bridging the deterministic evaluation of a process at a random time with its probabilistic expectation.

\begin{lemma}\label{T:stopping time}\cite[p.8]{MR1224450}
Let \( \tau\) be any stopping time, \( f \in L^1 \). Then \( f_\tau \) defined as \( f_{\tau(\omega)}(\omega) \), \( f_n = \mathbb E(f|\mathcal{F}_n) \), \( n \ge 0 \), \( f_\infty = f \), satisfies \(f_\tau= \mathbb E(f|\mathcal{F}_\tau)\) and \(\mathbb E(|f_\tau|) \leq \mathbb E(|f|).\)
\end{lemma}

The identity \( f_{\tau} = \mathbb{E}(f \mid \mathcal{F}_{\tau}) \) establishes a profound connection between pathwise behavior at stopping times and conditional expectations. This relationship underpins the development of more advanced operator estimates through stopping-time techniques. The systematic application of these methods yields sparse bounds for both uniformly bounded martingale transforms and Doob's maximal operator, with significant implications in weighted theory. 

To formalize the underlying structure enabling these sparse domination results, we introduce two key concepts: the locally filter-shifted stopping time, which provides the necessary technical framework, and conditional sparsity, which captures the essential structural property revealed through this approach. Together, these concepts form a coherent foundation for sparse operator theory in the martingale context.

\begin{definition}\label{def:locally filter-shifted}
Let $\tau$ be a stopping time on $(\Omega, \mathcal{F}, \mu, \{\mathcal{F}_n\}_{n\geq 0})$.

\begin{itemize}
\item If $\mu(\{\tau < \infty\}) > 0$, the local filtered probability space induced by $\tau$ is defined as the quadruple $(\widetilde{\Omega}, \widetilde{\mathcal{F}}, \widetilde{\mu}, \{\widetilde{\mathcal{F}}_n\}_{n\geq 0})$, where:
 \[
\widetilde{\Omega}=\{\tau < +\infty\},~\widetilde{\mathcal{F}} =\mathcal{F}|_{\{\tau < +\infty\}},~\widetilde{\mathcal{F}}_{n}=\mathcal{F}_{\tau+n}|_{\{\tau < +\infty\}},~\widetilde{\mu}= \frac{\mu|_{\{\tau < +\infty\}}}{\mu(\{\tau < +\infty\})}.
\]
  Any stopping time on this local space is called a locally filter-shifted stopping time relative to $\tau$.

\item If $\mu(\{\tau < \infty\}) = 0$, then the local filtered probability space induced by $\tau$ is defined to be the original space $(\Omega, \mathcal{F}, \mu, \{\mathcal{F}_n\}_{n\geq 0})$, and we define $\tau$ itself as the only locally filter-shifted stopping time relative to $\tau.$
\end{itemize}
\end{definition}

These two concepts allow us to characterize a probabilistic form of sparsity—termed conditional sparsity—which constitutes one of the central contributions of this work. Our framework employs an iterative definition beginning with Definition \ref{def:locally filter-shifted}. Its initial step sets $\tau \equiv 0$, which induces a local filtered probability space. This space is precisely the original filtered probability space itself, $(\Omega, \mathcal{F}, \mu, \{\mathcal{F}_n\}_{n\geq 0})$. Having established this base case, we now define the subsequent concept in Definition \ref{def:sparse_stopping_time}.

\begin{definition}\label{def:sparse_stopping_time}
Let $\{\tau^{(i)}\}_{i\geq 0}$ be a sequence of stopping times on the filtered probability spaces $(\Omega^{(i)}, \mathcal{F}^{(i)}, \mu^{(i)}, \{\mathcal{F}_n^{(i)}\}_{n\geq 0})$, where for  $i \geq 0$, $\tau^{(i+1)}$ is a locally filter-shifted stopping time relative to $\tau^{(i)}$, $\tau^{(0)}=0$ and  $(\Omega^{(0)}, \mathcal{F}^{(0)}, \mu^{(0)}, \{\mathcal{F}_n^{(0)}\}_{n\geq 0})=(\Omega, \mathcal{F}, \mu, \{\mathcal{F}_n\}_{n\geq 0}). $ The collection $\{\tau^{(i)}\}_{i\geq 0}$ is said to possess the property of conditional sparsity if there exists a constant $C\geq 1$ such that for all $i \geq 0$ we have
\[
 \mathbf{1}_{\{\tau^{(i)}<\infty\}}\leq C \mathbb{E}\left(\mathbf{1}_{\{\tau^{(i)}<\infty\} \setminus \{\tau^{(i+1)}<\infty\}} \mid \mathcal{F}^{(i)}_{\tau^{(i)}}\right)\mathbf{1}_{\{\tau^{(i)}<\infty\}}.
\]
\end{definition}

The property ensures that the stopping times are sufficiently sparse in a conditional sense, with the sets $\{\tau^{(i)}<\infty\} \setminus \{\tau^{(i+1)}<\infty\}$ forming an essential partition of the probability space. Intuitively, it guarantees that at each stopping index \(i\), the conditional expectation of the indicator of \(\{\tau^{(i)} < \infty\} \setminus \{\tau^{(i+1)} < \infty\}\) (the set where stopping occurs at \(\tau^{(i)}\) but not at \(\tau^{(i+1)}\)) is bounded below by a fixed fraction of the indicator of the event \(\{\tau^{(i)} < \infty\}\). This technical condition plays a crucial role in establishing the boundedness properties of the operators defined below. For convenience, we denote $E_i = \{\tau^{(i)} < \infty\}$ and $S_i=\{\tau^{(i)}<\infty\}\setminus \{\tau^{(i+1)}<\infty\}.$

\begin{definition}\label{def:sparse}
An operator \( S \) is called sparse if it admits the representation
\[
S|f| = \sum_{i=0}^{\infty}\mathbb{E}\left( |f|\mid \mathcal{F}^{(i)}_{\tau^{(i)}}\right)\mathbf{1}_{E_i},
\]
where $f\in L^1$ and  the collection $\{\tau^{(i)}\}_{i\geq 0}$ possesses the property of conditional sparsity.
\end{definition}

The representation integrates conditional expectations of \(|f|\) over the sequence of stopping times \(\{\tau^{(i)}\}_{i\geq 0}\). For each stopping time \(\tau^{(i)}\), the operator takes the conditional expectation of \(|f|\) with respect to the \(\sigma\)-algebra \(\mathcal{F}_{\tau^{(i)}}^{(i)}\) and multiplies it by the indicator of the event \(E_i\). The conditional sparsity property ensures that this infinite sum is analytically tractable.

\begin{definition}\label{def:simple_sparse}
An operator \( \mathbb{S} \) is called simple sparse if it admits the representation 
\[
\mathbb{S}|f| = \sum_{i=0}^{\infty}\mathbb{E}\left( |f|\mid \mathcal{F}^{(i)}_{\tau^{(i)}}\right)\mathbf{1}_{S_i},
\]
where $f\in L^1$ and  the collection $\{\tau^{(i)}\}_{i\geq 0}$ possesses the property of conditional sparsity.
\end{definition}

The fundamental advantage of this representation stems from the mutual disjointness of the sets $\{S_i\}_{i\geq 0}$. This disjointness ensures that for any $p \geq 1$, the $p$-th power of the sum equals the sum of the $p$-th powers without cross terms, which dramatically simplifies $L^p$-norm estimates.

\subsection{Weights} 
Let $(\Omega, \mathcal{F}, \mu, \{\mathcal{F}_n\}_{n \ge 0})$ be a filtered probability space. 
A weight $w$ is a positive random variable in $L^1.$ If \(A \in \mathcal{F}\), we denote $\int_A d\mu$ and $\int_A w d\mu$ by $\mu(A)$ and  $w(A),$
respectively.
For \(p > 1\), a martingale \(f = (f_n) \in L^p (w)\) is meant as \(f_n = \mathbb E(f | \mathcal{F}_n)\), \(f \in L^p (w)\). Without loss of generality, we normalize $w$ so that $w(\Omega) = 1$; otherwise, one may consider $w/w(\Omega)$.

Consider the weighted probability space $(\Omega, \mathcal{F}, wd\mu)$ and the corresponding filtered space $(\Omega, \mathcal{F}, wd\mu, \{\mathcal{F}_n\}_{n\geq 0})$. For a stopping time $\tau$, the conditional expectation with respect to $\mathcal{F}_{\tau}$ under the measure $wd\mu$ is well-defined and denoted by $\mathbb{E}_w(\cdot \mid \mathcal{F}_{\tau})$. 

To streamline subsequent expressions, we use the following notation
\begin{equation}\label{weighted_expectation}
\mathbbm{E} [\,\cdot \; w ]
  =\int_{\Omega}\cdot wd\mu,
   \end{equation}
and  for any stopping time $\tau$, \begin{equation}\label{weighted_sampling}
   \cdot_{\,\tau, w} =\mathbbm{E}_w [\, \cdot \;|\,
\mathcal{F}_{\tau}].
  \end{equation}
  
A fundamental relationship between conditional expectations under the original and weighted measures is given as follows. Take $A \in \mathcal{F}_{\tau}$ and $f \in L^1(\mu).$ Then
\[
\begin{aligned}
\int_A f_{\tau}d\mu &= \int_A f d\mu = \int_A f w^{-1} w d\mu = \int_A \mathbb{E}_w(fw^{-1} | \mathcal{F}_{\tau}) wd\mu\\
&= \int_A \mathbb{E}_w(fw^{-1} | \mathcal{F}_{\tau}) w_{\tau} d\mu.
\end{aligned}
\]
Since $A$ is  arbitrary in $\mathcal{F}_{\tau},$ it follows that $\mathbb{E}_w(fw^{-1} | \mathcal{F}_{\tau}) w_{\tau}=f_{\tau}.$

We begin by reviewing the foundational concept of  the $A_p$ weight and formalizing the notion of mixed weights.

\begin{definition}\footnote{Within weighted martingale theory, the integrability conditions on $w$ and $\sigma$ are standard; see \cite[p.~248]{MR1224450}.}
Let $w$ be a weight and let \( 1 < p < \infty \). Suppose $\sigma=w^{-\frac{1}{p-1}} \in L^1.$ For $n\in \mathbb N,$ we define
\[A_{n,p}(w)= \left\| w_{n} \sigma^{p - 1}_{n} \right\|_{\infty}.\]
We define the $A_p$ characteristic of the weight $w$ by
\[ [w]_{A_p} = \sup_{n\geq0}A_{n,p}(w).\] 
If $[w]_{A_p}$ is finite, then we say $w \in A_p.$
\end{definition}

\begin{definition}
Let $\alpha,~\beta\geq0$ and $\nu\geq1.$
For $\nu+1\leq p\leq r<+\infty,$
we define the $A_p^{\alpha}A_r^{\beta}$ characteristic of the weight $w$ by
\[ [w]_{A_p^{\alpha}A_r^{\beta}} = \sup_{n\geq0}(A_{n,p}(w))^{\alpha} (A_{n,r}(w))^{\beta}.\]
If $[w]_{A_p^{\alpha}A_r^{\beta}}$ is finite, then we say $w \in A_p^{\alpha}A_r^{\beta}$.
\end{definition}

\begin{remark}\label{rk:Hereditarily}Let $\tau$ be a stopping time on $(\Omega, \mathcal{F}, \mu, \{\mathcal{F}_n\}_{n \ge 0}).$ If $w\in A_p,$ then we define
\[A_{\tau,p}(w)= \left\| w_{\tau} \sigma^{p - 1}_{\tau} \right\|_{\infty}.\] Because 
\begin{eqnarray*}
w_{\tau} \sigma^{p - 1}_{\tau}&=&w \sigma^{p - 1}\mathbf 1_{\{\tau=\infty\}}+\sum\limits_{n=0}^{\infty}w_{\tau} \sigma^{p - 1}_{\tau}\mathbf 1_{\{\tau=n\}}\\
&=&\lim_{n\rightarrow\infty}w_n \sigma_n^{p - 1}\mathbf 1_{\{\tau=\infty\}}+\sum\limits_{n=0}^{\infty}w_{n} \sigma^{p - 1}_{n}\mathbf 1_{\{\tau=n\}}\\
&\leq&[w]_{A_p},
\end{eqnarray*}
we have $A_{\tau,p}(w)\leq[w]_{A_p}.$ Let $w \in A_p^{\alpha}A_r^{\beta}$. Similarly, we have $(A_{\tau,p}(w))^{\alpha} (A_{\tau,r}(w))^{\beta}\leq[w]_{A_p^{\alpha}A_r^{\beta}}.$
\end{remark}

Next, we recall the definitions of the two-weight conditions $S_p^*$ and $B_p.$
\begin{definition}\label{definition Sp}Let $u \hbox{ and }w$ be weights and $1<p<\infty.$
Denote $\sigma=w^{-\frac{1}{p-1}}\in L^1.$
We say that the couple of weights $(u,w)$
satisfies the condition $S^*_p,$ if
      \begin{equation}\label{bi-SP}
[u,w]_{S^*_p}=\sup\limits_{n\in \mathbb N}
\Big(\frac{\int_E {^*M_n}(\sigma)^pud\mu}{\sigma({E})}\Big)^{\frac{1}{p}}<\infty.
      \end{equation}
\end{definition}

\begin{definition}\label{definition Bp}Let $u\hbox{ and }w$ be weights and $1<p<\infty.$
Denote that $\sigma=w^{-\frac{1}{p-1}}\in L^1.$
We say that the couple of weights $(u,w)$
satisfies the condition $B_p,$ if
there exists a positive constant $C$ such that for all $n\in \mathbb N$ we have
      \begin{equation}\label{bi-BP}\mathbb E_n(u)\mathbb E_n(\sigma)^p\leq
C\exp\big(\mathbb E_n(\log\sigma)\big).
      \end{equation}
We denote by $[u,w]_{B_p}$
the smallest constant $C$ in \eqref{bi-BP}.
\end{definition}

We conclude by recalling the $[w]_{A^*_{\infty}}$ condition and formally defining $(A_{p'})^{\frac{1}{p'}}(A^*_{\infty})^{\frac{1}{p}}.$

\begin{definition}\label{definition Wp}Let $w$ be a weight.
We say that the weight $w$
satisfies the condition $A^*_{\infty},$ if
there exists a positive constant $C$ such that for all $n\in \mathbb N$ we have
      \begin{equation}\label{bi-WP}
         \mathbb E_n({^*M_nw})
\leq C w_n.
      \end{equation}
We denote by $[w]_{A^*_{\infty}}$
the smallest constant $C$ in \eqref{bi-WP}.
\end{definition}

\begin{remark}\label{rem1}
We summarize basic properties about the conditions.
Let $w\in A_p$ and $\sigma=w^{1-p'}.$ Then
\begin{enumerate}
  \item $\sigma\in A_{p'}$ and $[\sigma]^{\frac{1}{p'}}_{A_{p'}}=[w]^{\frac{1}{p}}_{A_p};$
  \item \cite[Theorem1.8] {MR4745061} $w\in A^*_{\infty}$ and $[w]_{A^*_{\infty}}\lesssim[w]_{A_p}.$
\end{enumerate}
\end{remark}

\begin{definition}\label{df:mixed_Ap_infty}Let $w$ be a weight and let \( 1 < p < \infty \). Suppose $\sigma=w^{-\frac{1}{p-1}} \in L^1.$ 
We define the mixed condition $(A_{p'})^{\frac{1}{p'}}(A^*_{\infty})^{\frac{1}{p}}$ by
      \begin{equation}\label{mixcon1}
      [\sigma]_{(A_{p'})^{\frac{1}{p'}}(A^*_{\infty})^{\frac{1}{p}}}=
      \sup\limits_{n\in \mathbb N}\Big(\sigma_n ^{p-1} w_n\cdot(\frac{{^*M_n\sigma}}{\sigma_n})_{n}\Big)^{\frac{1}{p}}.
      \end{equation}
\end{definition}

\section{Uniformly Bounded Martingale Transforms}\label{sec:mar_trans}
This section establishes sparse domination for uniformly bounded martingale transforms through mathematical induction. The inductive construction progressively builds locally filter-shifted stopping times, yielding a family of stopping times that possesses the property of conditional sparsity and generates the desired sparse operators.

\subsection{Sparse Domination for Uniformly Bounded Martingale Transforms}\label{subsec:mar_trans}
In the subsequent proof, stopping times emerge as the fundamental elements of sparse theory. The sparse operators constructed through these stopping times derive their sparse character by virtue of conditional expectations.

\begin{lemma}\label{lem:essential}Let $(\Omega, \mathcal{F}, \mu, \{\mathcal{F}_n\}_{n \ge 0})$ be a filtered probability space. There is a constant \( C > 0 \) so that for all  \( f \in L^1 \) and uniformly bounded martingale transforms \( T \) corresponding to a multiplier sequence \( v = (v_n)_{n \geq 0} \) with \( \sup_n\|v_n\|_{\infty} \leq 1 \), there is a stopping time $\tau$ on $(\Omega, \mathcal{F}, \mu, \{\mathcal{F}_n\}_{n \ge 0})$ such that 
\begin{equation}\label{eq:mh_begin}
|Tf|\leq 12|f|_0+|f|_{\tau}\mathbf 1_{\{\tau<\infty\}}+|T^{\tau}f|,
\end{equation}
where $T^{\tau}f=(\sum\limits_{n=\tau+1}^{\infty}v_{n-1} \Delta_n f) \mathbf 1_{\{\tau<\infty\}}.$
Furthermore, $\tau$ has the following properties
\begin{itemize}
    \item[(\ref{lem:essential}a)]  \label{qasqq}$\mu(\tau<\infty)\leq \frac{1}{2}\mu(\Omega).$
    \item[(\ref{lem:essential}b)]  For all $A \subset \Omega, A \in \mathcal{F}_0$ there holds
$\mathbf 1_{A}\leq 2\mathbb E(\mathbf 1_{E(A)}|\mathcal{F}_{0})\mathbf 1_{A},$
where $E(A)=A\backslash \{\tau<\infty\}=A\cap\{\tau=\infty\}.$
\end{itemize}
\end{lemma}

\begin{proof}[Proof of Lemma \ref{lem:essential}] The case $\|f\|_1=0$ is trivial with $\tau\equiv\infty$ and $T^{\tau}f=0.$
Let $\|f\|_1>0$. For a uniformly integral martingale
 $f = (f_n)_{n \ge 0}$ on $(\Omega, \mathcal{F}, \mu, \{\mathcal{F}_n\}_{n \ge 0})$ with $f_n=\mathbb E(f|\mathcal{F}_n),$ we
set
\begin{equation}
\label{eq:set_definition}
E = \{\max\{Mf, M(Tf)\} > C|f|_0\},
\end{equation}
where the constant \( C \) to be determined later. 
Let \begin{equation}\label{df:tau}\tau=\inf\{n:|f_n|\vee |(Tf)_n|>C|f|_0\}.\end{equation} 
Then $E=\{\tau<\infty\}.$ 
Denoting $\Omega_{0}=\{|f|_0=0\}$ and $\Omega_{1}=\{|f|_0\neq0\},$ we write 
\[\widetilde{\Omega}=\Omega_1,~\widetilde{\mathcal{F}} =\mathcal{F}|_{\Omega_1},~\widetilde{\mathcal{F}_n} =\mathcal{F}_n|_{\Omega_1},~\widetilde{\mu}= \frac{\mu|_{\Omega_{1}}}{\mu(\Omega_{1})}\] 
and \[\widetilde{f_n} =\frac{f_n|_{\Omega_1}}{|f|_0|_{\Omega_1}},~n\geq 0.\]
 It follows that $(\widetilde{f_n})_{n\geq 0}$ is a uniformly integral martingale on $(\widetilde{\Omega}, \widetilde{\mathcal{F}}, \widetilde{\mu}, 
\{\widetilde{\mathcal{F}_n}\}_{n \ge 0})$ and \[\lim_{n\rightarrow+\infty}\widetilde{f_n}=\frac{f|_{\Omega_1}}{|f|_0|_{\Omega_1}},~a.e.\] Here, $\frac{f|_{\Omega_1}}{|f|_0|_{\Omega_1}}$ is denoted by $\widetilde{f}.$
Recalling Remark \ref{re: uniform}, we have $\widetilde{f}=(\widetilde{f_n})_{n\geq 0}.$ 
Letting $\widetilde{\tau}=\inf\{n:|\widetilde{f_n}|\vee |(T\widetilde{f})_n|>C\},$ we deduce that
$\tau|_{\Omega_1}=\widetilde{\tau}$ and 
\[
E= \left\{ \max\{ M \widetilde{f}, M(T\widetilde{f })\} > C\right\}=\left\{M \widetilde{f}>C\right\}\cup\left\{M(T \widetilde{f})>C\right\}.
\]
Thus
\[
\begin{aligned}
\widetilde{\mu}(E)
&\leq \widetilde{\mu}(\left\{M \widetilde{f}>C\right\})+\widetilde{\mu}(\left\{M(T \widetilde{f})>C\right\}) \\
&\leq\frac{ 3}{C} \| \widetilde{f}\|_{L^1(\widetilde{\Omega}, \widetilde{\mathcal{F}}, \widetilde{\mu})}=\frac{ 3}{C},
\end{aligned}
\]
where we have used Lemma \ref{lem: weak} and  Corollary \ref{c:weak 1 1}.  
Setting $C=6,$ we get $\widetilde{\mu}(E)\leq \frac{1}{2}$.

The case $\widetilde{\mu}(E)=0$ is trivial with $\tau\equiv\infty$ and $T^{\tau}f=0.$ 

We prove the case $0<\widetilde{\mu}(E)$ as follows.  Because $\widetilde{\mu}(E)\leq \frac{1}{2},$ we have
$\mu(E)\leq\frac{1}{2}\mu(\Omega_{1})\leq\frac{1}{2}\mu(\Omega).$  Then $\mu(\tau<\infty)\leq \frac{1}{2}\mu(\Omega),$
which is (\ref{lem:essential}a) and equivalent to 
\[\mu(\Omega)\leq2\mu(\Omega\backslash\{\tau<\infty\})=2\mu(\Omega\cap\{\tau=\infty\}).\]

Let $A\in \mathcal{F}_0$.
Denoting $E(A)=A\backslash \{\tau<\infty\}=A\cap\{\tau=\infty\},$ we have
$$  \mu(A) \leqslant  2 \mu(E(A)),$$
by repeating the above proof and replacing $\Omega$ by $A.$
For any $A$ and $A_1\in \mathcal{F}_0,$ we get  
\begin{eqnarray*}
\int_{A_1}\mathbf 1_{A}d\mu&=& \mu(A_1\cap A)\\
&\leq&2 \mu\big(E(A_1\cap A)\big)\\
&=&2 \mu\big(A_1\cap A\cap\{\tau=\infty\}\big)\\
&=&2 \mu\big(A_1\cap E(A)\big)\\
&=&2\int_{A_1}\mathbf 1_{E(A)}d\mu\\
&=&2\int_{A_1}\mathbb E(\mathbf 1_{E(A)}|\mathcal{F}_0)d\mu.
\end{eqnarray*}
Since $A_1$ is arbitrary, we obtain $\mathbf 1_{A}\leq 2\mathbb E(\mathbf 1_{E(A)}|\mathcal{F}_0)\mathbf 1_{A}.$ This is (\ref{lem:essential}b).

For the above $\tau,$ we claim that
\[|Tf|\leq 12|f|_0+|f|_{\tau}\mathbf 1_{\{\tau<\infty\}}+|T^{\tau}f|,\]
where $T^{\tau}f=(\sum\limits_{n=\tau+1}^{\infty}v_{n-1} \Delta_n f) \mathbf 1_{\{\tau<\infty\}}.$ 

Let $x\in \Omega\setminus E= \{\tau=\infty\},$ then $Mf(x)\vee M(Tf)(x)\leq6|f|_0(x).$ It follows that
\begin{equation}\label{eq:setminus}
|Tf|\mathbf 1_{\Omega\setminus E}\leq6|f|_0\mathbf 1_{\Omega\setminus E}.
\end{equation}

Let $x\in E.$ Then $x\in \{\tau<\infty\}.$ Suppose $\tau(x)=n_0.$ It follows that
\[
\begin{aligned}
&|Tf|\mathbf 1_{\{\tau(x)=n_0\}}\\
&=|\sum\limits_{n=1}^{n_0}v_{n-1} \Delta_n f+\sum\limits_{n=n_0+1}^{\infty}v_{n-1} \Delta_n f|\mathbf 1_{\{\tau(x)=n_0\}}\\
&=|\sum\limits_{n=1}^{n_0-1}v_{n-1} \Delta_n f-v_{n_0-1}  f _{n_0-1}+v_{n_0-1} f _{n_0}+\sum\limits_{n=n_0+1}^{\infty}v_{n-1} \Delta_n f|\mathbf 1_{\{\tau(x)=n_0\}}\\
&\leq|\sum\limits_{n=1}^{n_0-1}v_{n-1} \Delta_n f|\mathbf 1_{\{\tau(x)=n_0\}}+|v_{n_0-1}  f _{n_0-1}|\mathbf 1_{\{\tau_0(x)=n_0\}}+\\&\quad+|v_{n_0-1} f _{n_0}+\sum\limits_{n=n_0+1}^{\infty}v_{n-1} \Delta_n f|\mathbf 1_{\{\tau(x)=n_0\}}.
\end{aligned}
\]
From the definition of $\tau,$ we have $$|\sum\limits_{n=1}^{n_0-1}v_{n-1}(x) \Delta_n f(x)|\mathbf 1_{\{\tau(x)=n_0\}}\leq6|f|_0(x)\mathbf 1_{\{\tau(x)=n_0\}}$$ and $$|v_{n_0-1}(x)  f _{n_0-1}(x)|\mathbf 1_{\{\tau(x)=n_0\}}\leq6|f|_0(x)\mathbf 1_{\{\tau(x)=n_0\}}.$$ Thus
\[
\begin{aligned}
|Tf(x)|\mathbf 1_{E}&=12|f|_0(x)\mathbf 1_{E}+|v_{\tau-1}(x) f _{\tau}(x)+\sum\limits_{n=\tau+1}^{\infty}v_{n-1}(x) \Delta_n f(x)|\mathbf 1_{E}.
\end{aligned}
\]
Combining this with \eqref{eq:setminus}, we derive
\[|Tf|\leq 12|f|_0+|f|_{\tau}\mathbf 1_{\{\tau<\infty\}}+|T^{\tau}f|,\]
where $T^{\tau}f=(\sum\limits_{n=\tau+1}^{\infty}v_{n-1} \Delta_n f) \mathbf 1_{\{\tau<\infty\}}.$ 
\end{proof}

\begin{lemma}\label{lem:induce}Let $l\in \mathbb {N}.$ For all  \( f \in L^1(\mu) \) and uniformly bounded martingale transforms \( T \) corresponding to a multiplier sequence \( v = (v_n)_{n \geq 0} \) with \( \sup_n\|v_n\|_{\infty} \leq 1 \), there are stopping times $\tau^{(m)}$ on filtered probability spaces $(\Omega^{(m)}, \mathcal{F}^{(m)}, \mu^{(m)}, \{\mathcal{F}^{(m)}_n\}_{n \ge 0})$ such that 
\begin{equation}\label{eq:lem_induce_sparse}
|Tf|\leq 12|f|_0+14\sum\limits_{m=1}^{l-1}|f|_{\tau^{(m)}}\mathbf 1_{\{\tau^{(m)}<\infty\}}+|f|_{\tau^{(l)}}\mathbf 1_{\{\tau^{(l)}<\infty\}}+|T^{\tau^{(l)}}f|,
\end{equation}
where $T^{\tau^{(l)}}f=(\sum\limits_{n=\tau^{(l)}+1}^{\infty}v_{n-1} \Delta_n f) \mathbf 1_{\{\tau^{(l)}<\infty\}}.$ 
Furthermore, $\tau^{(m)}$ has the following properties
\begin{equation}\label{eq:sparse}
\mathbf 1_{\{\tau^{(m-1)}<\infty\}}\leq 2\mathbb E(\mathbf 1_{\{\tau^{(m-1)}<\infty\}\backslash\{\tau^{(m)}<\infty\}}|\mathcal{F}_{\tau^{(m-1)}})\mathbf 1_{\{\tau^{(m-1)}<\infty\}},
\end{equation}
where $\tau^{(0)}$ is understood as $0.$
\end{lemma}

\begin{proof}[Proof of Lemma \ref{lem:induce}] 
We proceed by mathematical induction on \( l \).

\paragraph{Base Case (\( l = 1 \))}
When $l=1,$ we can prove the case by 
Lemma \ref{lem:essential}. Letting $$(\Omega^{(1)}, \mathcal{F}^{(1)}, \mu^{(1)}, \{\mathcal{F}^{(1)}_n\}_{n \ge 0}) =(\Omega, \mathcal{F}, \mu, \{\mathcal{F}_n\}_{n \ge 0})$$ and $\tau^{(1)}=\tau,$ we obtain that  
\[|Tf|\leq 12|f|_0+|f|_{\tau^{(1)}}\mathbf 1_{\{\tau^{(1)}<\infty\}}+|T^{\tau^{(1)}}f|\]
with $T^{\tau^{(1)}}f=(\sum\limits_{n=\tau^{(1)}+1}^{\infty}v_{n-1} \Delta_n f) \mathbf 1_{\{\tau^{(1)}<\infty\}}.$ This implies that
\[|Tf|\leq 12|f|_0+14\sum\limits_{m=1}^{l-1}|f|_{\tau^{(m)}}\mathbf 1_{\{\tau^{(m)}<\infty\}}+|f|_{\tau^{(l)}}\mathbf 1_{\{\tau^{(l)}<\infty\}}+|T^{\tau^{(l)}}f|.\]
Using (\ref{lem:essential}b), we obtain  
$$\mathbf 1_{\{\tau^{(0)}<\infty\}}\leq 2\mathbb E(\mathbf 1_{\{\tau^{(0)}<\infty\}\backslash\{\tau^{(1)}<\infty\}}|\mathcal{F}_{\tau^{(0)}})\mathbf 1_{\{\tau^{(0)}<\infty\}}.$$ 

\paragraph{Inductive Hypothesis (\(l=j\))}
Assume inductively that stopping times $\tau^{(j)}$  have been constructed. 

\paragraph{Inductive Step (\( l = j + 1 \))}If $\mu\{\tau^{(j)}<\infty\}=0,$ then we define $\tau^{(j+1)}=\tau^{(j)}.$ 
Otherwise, we define the quadruple 
\begin{eqnarray*}
&&(\Omega^{(j+1)}, \mathcal{F}^{(j+1)}, \mu^{(j+1)}, \{\mathcal{F}^{(j+1)}_n\}_{n \ge 0}) \\
&=&(\{\tau^{(j)}<\infty\}, \mathcal{F}|_{\{\tau^{(j)}<\infty\}}, \frac{\mu|_{\{\tau^{(j)}<\infty\}}}{\mu\{\tau^{(j)}<\infty\}}, \{\mathcal{F}^{(j)}_{\tau^{(j)}+n}|_{\{\tau^{(j)}<\infty\}}\}_{n \ge 0}).  
\end{eqnarray*} 
Applying Lemma \ref{lem:essential} to this quadruple, we obtain a stopping time, which we denote by $\tau^{(j+1)}.$ It follows that
\[|T^{\tau^{(j)}}f|\leq 12|f|_{\tau^{(j)}}\mathbf 1_{\{\tau^{(j)}<\infty\}}+|f|_{\tau^{(j+1)}}\mathbf 1_{\{\tau^{(j+1)}<\infty\}}+|T^{\tau^{(j+1)}}f|\]
where $$T^{\tau^{(j+1)}}f=(\sum\limits_{n=\tau^{(j+1)}+1}^{\infty}v_{n-1} \Delta_n f) \mathbf 1_{\{\tau^{(j+1)}<\infty\}}$$ 
and $$\mathbf 1_{\{\tau^{(j)}<\infty\}}\leq 2\mathbb E(\mathbf 1_{\{\tau^{(j)}<\infty\}\backslash\{\tau^{(j+1)}<\infty\}}|\mathcal{F}_{\tau^{(j)}})\mathbf 1_{\{\tau^{(j)}<\infty\}}.$$ Using the inductive hypothesis, we have
\[|Tf|\leq 12|f|_0+14\sum\limits_{m=1}^{j}|f|_{\tau^{(m)}}\mathbf 1_{\{\tau^{(m)}<\infty\}}+|f|_{\tau^{(j+1)}}\mathbf 1_{\{\tau^{(j+1)}<\infty\}}+|T^{\tau^{(j+1)}}f|.\]

\paragraph{Conclusion}
This completes the proof by induction.
\end{proof}

\begin{proof}[Proof of Theorem \ref{thm:sparse}] 
Following Lemma \ref{lem:induce}, we have $\lim\limits_{m\rightarrow\infty}\mu(\{\tau^{(m)}<\infty\})=0.$ Then
\[|Tf|\leq 12|f|_0+14\sum\limits_{m=1}^{\infty}|f|_{\tau^{(m)}}\mathbf 1_{\{\tau^{(m)}<\infty\}}\leq 17\sum\limits_{m=0}^{\infty}|f|_{\tau^{(m)}}\mathbf 1_{\{\tau^{(m)}<\infty\}}.\] This is \eqref{eq:main} with $S|f|=\sum\limits_{m=0}^{\infty}|f|_{\tau^{(m)}}\mathbf 1_{\{\tau^{(m)}<\infty\}}$  and $C=17.$ \footnote{No attempt is made in this proof to determine the optimal constant $C$. The number $17$ is employed as a convenient prime.}
\end{proof}

\subsection{Weighted Inequalities for Uniformly Bounded Martingale Transforms}

The estimation of weighted norms for the sparse operators is achieved through duality theory combined with a Parseval-type identity for conditional expectations. This approach transforms the norm estimation problem into a dual framework where the structural properties of conditional expectations---specifically their role as self-adjoint operators in the context of \( L^p \)-\( L^{p'} \) duality---can be systematically applied.

\begin{proof}[Proof of Theorem \ref{Theorem_SXC}] Let $\sigma=w^{-\frac{1}{p-1}}$, then $\sigma^p w = \sigma.$ Because $$\| S|f|\|_{L^p (w)}
\lesssim  [w]_{A_p} ^{\max (1, 1 / (p - 1))} \| f \|_{L^p (w)}$$ is equivalent to
\begin{equation}\label{Theorem_IA} \| S (| f | \sigma )\|_{L^p (w)} \lesssim [w]_{A_p} ^{\max (1, 1 / (p - 1))} \| f
   \|_{L^p (\sigma)} ,\end{equation}
we will prove \eqref{Theorem_IA}.

In view of duality, it suffices to estimate

\[ \mathbbm{E} \left[ \sum\limits_{m=0}^{+\infty} (| f | \sigma)_{\tau^{(m)}} \mathbf 1_{E_m} hw \right] ,\]
 where $h$ is a positive function in ${ L^{p^\prime}(w)}$ and $E_m=\{\tau^{(m)}<\infty\}$. Following from the fundamental relationship between conditional expectations under the original and weighted measures, we have  
 \[ \mathbbm{E} \left[ \sum\limits_{m=0}^{+\infty} (| f | \sigma)_{\tau^{(m)}} \mathbf 1_{E_m} hw \right]
   =\mathbbm{E} \left[ \sum\limits_{m=0}^{+\infty} | f |_{\tau^{(m)}, \sigma}  h _{\tau^{(m)}, w} \mathbf 1_{E_m} \sigma_{\tau^{(m)}}
   w_{\tau^{(m)}} \right] .\]

    Case $1:$ $p\geq2.$  We have $\max (p, p^\prime)=p$ and we aim at a constant
\[ [w]_{A_p} ^{\max (1, 1 / (p - 1))} =[w]_{A_p} = \sup_{n} \| w_{n}
   \sigma^{p-1}_{n} \|_{\infty}. \]
We continue the above calculation
\begin{eqnarray*}
\mathbbm{E} \left[ \sum\limits_{m=0}^{+\infty} (| f | \sigma)_{\tau^{(m)}} \mathbf 1_{E_m}  h w \right]
  & \leqslant& [w]_{A_p}  \mathbbm{E} \left[ \sum\limits_{m=0}^{+\infty} | f
   |_{\tau^{(m)}, \sigma} h _{\tau^{(m)}, w} \mathbf 1_{E_m} \sigma^{2 - p}_{\tau^{(m)}}  \right]\\
  &=& [w]_{A_p} \mathbbm{E} \left[ \sum\limits_{m=0}^{+\infty} | f
   |_{\tau^{(m)}, \sigma} h _{\tau^{(m)}, w} (\mathbf 1_{E_m})^{p-1}\sigma^{2 - p}_{\tau^{(m)}}  \right] .
\end{eqnarray*}
Using the conditional sparsity and Hölder’s inequality of the conditional expectation, we have
  \begin{eqnarray*}
\mathbf 1_{E_m}^{p-1}&\leq&2^{p-1}(\mathbf 1_{S_m})^{p-1}_{\tau^{(m)}}\mathbf 1_{E_m}\\
   &=&2^{p-1} (w^{\frac{1}{p}}\sigma^{\frac{1}{ p'}}\mathbf 1_{S_m})^{p-1}_{\tau^{(m)}}\mathbf 1_{E_m}\\
    &\leq&2^{p-1} (w\mathbf 1_{S_m})^{\frac{1}{ p'}}_{\tau^{(m)}}
    (\sigma\mathbf 1_{S_m})^{\frac{p-1}{p^\prime}}_{\tau^{(m)}}\mathbf 1_{E_m}.
    \end{eqnarray*}

Since $2-p\leq0,$ we have $\sigma^{2-p}_{\tau^{(m)}} \leq(\sigma\mathbf 1_{S_m})^{2-p}_{\tau^{(m)}} $ on  $E_m.$
Thus
   \begin{eqnarray*}
&~& \mathbbm{E}  \left[ \sum\limits_{m=0}^{+\infty} (| f | \sigma)_{\tau^{(m)}} \mathbf 1_{E_m}h w \right] \\
 &\leq & [w]_{A_p}  \mathbbm{E} \left[ \sum\limits_{m=0}^{+\infty} | f|_{\tau^{(m)}, \sigma} h_{\tau^{(m)}, w} (\mathbf 1_{E_m})^{p-1}_{\tau^{(m)}}(\sigma\mathbf 1_{S_m})^{2 - p}_{\tau^{(m)}}  \right] \\
  &\lesssim& [w]_{A_p}  \mathbbm{E} \left[ \sum\limits_{m=0}^{+\infty} | f
   |_{\tau^{(m)}, \sigma} h_{\tau^{(m)}, w}(w\mathbf 1_{S_m})^{\frac{1}{ p'}}_{\tau^{(m)}}
    (\sigma\mathbf 1_{S_m})^{\frac{p-1}{p^\prime}}_{\tau^{(m)}}\mathbf 1_{E_m}(\sigma\mathbf 1_{S_m})^{2-p}_{\tau^{(m)}}  \right]
   \\
  &=& [w]_{A_p}  \mathbbm{E} \left[ \sum\limits_{m=0}^{+\infty} | f
   |_{\tau^{(m)}, \sigma}h_{\tau^{(m)}, w} (w\mathbf 1_{S_m})^{\frac{1}{p'}}_{\tau^{(m)}}(\sigma\mathbf 1_{S_m})^{\frac{1}{p}}_{\tau^{(m)}}  \right] .
   \end{eqnarray*}
Using Hölder’s inequality, we have
      \begin{eqnarray*} 
&~&\mathbbm{E}  \left[ \sum\limits_{m=0}^{+\infty} (|f| \sigma)_{\tau^{(m)}} \mathbf 1_{E_m} hw \right]\\
  & \leqslant& [w]_{A_p} \left[ \sum\limits_{m=0}^{+\infty}\mathbbm{E}  | f
   |_{\tau^{(m)}, \sigma}^p(\sigma\mathbf 1_{S_m})_{\tau^{(m)}}  \right]^{\frac{1}{p}}
   \left[ \sum\limits_{m=0}^{+\infty} \mathbbm{E}h^{p^\prime}_{\tau^{(m)}, w} (w\mathbf 1_{S_m})_{\tau^{(m)}} \right]^{\frac{1}{p^\prime}}\\
  &=& [w]_{A_p} \left[ \sum\limits_{m=0}^{+\infty}\mathbbm{E}  | f
   |_{\tau^{(m)}, \sigma}^p\sigma\mathbf 1_{S_m} \right]^{\frac{1}{p}}
   \left[ \sum\limits_{m=0}^{+\infty} \mathbbm{E}h^{p^\prime}_{\tau^{(m)}, w} w\mathbf 1_{S_m}\right]^{\frac{1}{p^\prime}}\\
  &\leqslant& [w]_{A_p} \left[ \sum\limits_{m=0}^{+\infty}\mathbbm{E}(M_{ \sigma}f) ^p\sigma\mathbf 1_{S_m} \right]^{\frac{1}{p}}
   \left[ \sum\limits_{m=0}^{+\infty} \mathbbm{E}(M_{w}h)^{p^\prime}w\mathbf 1_{S_m}\right]^{\frac{1}{p^\prime}}\\
  &\leqslant& [w]_{A_p}  \left[\mathbbm{E}(M_{\sigma}f )^p\sigma \right]^{\frac{1}{p}}
   \left[\mathbbm{E}(M_{w}h)^{p^\prime}w\right]^{\frac{1}{p^\prime}},
\end{eqnarray*}
where the maximal functions are taken with respect to
 weighted measure:
\begin{equation}\label{notation_weightedmax}
M_{\sigma}{f}=\sup_{n}| f|_{n,\sigma},~M_{w}{h}=\sup_{n}h_{n,w}.
\end{equation}
In virtue of the boundedness of Doob's maximal operator, we obtain that
 \begin{equation} \label{case1}\mathbbm{E}  \left[ \sum\limits_{m=0}^{+\infty} (|f| \sigma)_{\tau^{(m)}} \mathbf 1_{E_m} hw \right] \lesssim [w]_{A_p}  \| f
   \|_{L^p (\sigma)} \| h\|_{L^{p^\prime} (w)}.
 \end{equation}
It follows that
\begin{equation*} \| S (|f| \sigma )\|_{L^p (w)} \lesssim [w]_{A_p} ^{\max (1, 1 / (p - 1))} \| f
   \|_{L^p (\sigma)}. \end{equation*}

 Case $2:$ $1<p\leqslant2.$   We have $\max (p, p^\prime)=p^\prime\geq2$, and we aim at a constant
\begin{eqnarray*}
[w]_{A_p} ^{\max (1, 1 / (p - 1))}
&=&[w]_{A_p} ^{1 / (p - 1)} \\
&=& \sup_{n} \| w_{n}\sigma^{p-1}_{n} \|_{\infty} ^{1 / (p - 1)} \\
&=& \sup_{n} \| \sigma_{n}w_{n}^{1 / (p - 1)} \|_{\infty}  \\
&=& \sup_{n} \| \sigma_{n}w_{n}^{p^\prime - 1} \|_{\infty} \\
&=&\|\sigma\|_{A_{p^\prime}} .
\end{eqnarray*}
Then
\begin{eqnarray*} \mathbbm{E} \left[ \sum\limits_{m=0}^{+\infty} (|f| \sigma)_{\tau^{(m)}} \mathbf 1_{E_m} h w \right] \end{eqnarray*}
becomes \begin{eqnarray*}
  \mathbbm{E} \left[ \sum\limits_{m=0}^{+\infty}  (hw)_{\tau^{(m)}}\mathbf 1_{E_m} |f| \sigma \right].
\end{eqnarray*}
It follows from \eqref{case1} of Case $1$ that
\begin{eqnarray*}
  \mathbbm{E} \left[ \sum\limits_{m=0}^{+\infty}  (hw)_{\tau^{(m)}}\mathbf 1_{E_m} |f| \sigma\right]
  \lesssim \|\sigma\|_{A_{p^\prime}} \| h
   \|_{L^{p^\prime} (w)} \| f \|_{L^{p} (\sigma)}.
   \end{eqnarray*}
Thus we have
\begin{eqnarray*} \mathbbm{E} \left[ \sum\limits_{m=0}^{+\infty} (|f| \sigma)_{\tau^{(m)}} \mathbf 1_{E_m} h w \right]
  \lesssim  [w]_{A_p}^{1 / (p - 1)}  \| f \|_{L^{p} (\sigma)} \| h\|_{L^{p^\prime} (w)}
\end{eqnarray*}
which implies
\begin{eqnarray*}\| S (|f| \sigma) \|_{L^p (w)} &\lesssim&  [w]_{A_p} ^{\max (1, 1 / (p - 1))} \| f
   \|_{L^p (\sigma)} .\end{eqnarray*}
\end{proof}

\begin{proof}[Proof of Theorem \ref{Theorem_SM}]
We set $\sigma$ the dual weight of $w$ in the sense $\sigma^p w = \sigma$. So $\| S_{\nu} | f | \|_{L^p (w)}
\lesssim  [w]_{(A_p)^{\frac{1}{p-1}} (A_r)^{\frac{1}{\nu}-\frac{1}{p-1}}} \| f \|_{L^p (w)}$ becomes
\begin{eqnarray*}
\| S_{\nu} (| f | \sigma) \|_{L^p (w)} &\lesssim&  [w]_{(A_p)^{\frac{1}{p-1}} (A_r)^{\frac{1}{\nu}-\frac{1}{p-1}}}  \| | f | \sigma \|_{L^p (w)} \\ &=& [w]_{(A_p)^{\frac{1}{p-1}} (A_r)^{\frac{1}{\nu}-\frac{1}{p-1}}}  \| f\|_{L^p (\sigma)} . \end{eqnarray*}
Because there exists a function $h\geq0$ with $\|h\|_{ L^{{(p/\nu)}’}(w)}=1$ such that
\begin{eqnarray*}
\| S_{\nu} (| f | \sigma) \|_{L^p (w)} &=&\|S_{\nu}( | f |\sigma )\|_{L^{\nu} (hw)} ,
   \end{eqnarray*}
it suffices to estimate $\|S_{\nu} (| f |\sigma) \|_{L^{\nu} (hw)}.$
We obtain that

\begin{eqnarray} &~&\mathbbm{E} \left[ \sum\limits_{m=0}^{+\infty} (| f |\sigma)^{\nu}_{\tau^{(m)}} | h |_{\tau^{(m)}, w} \mathbf 1_{E_m} w_{\tau^{(m)}} \right]\\
  & =&\mathbbm{E} \left[\sum\limits_{m=0}^{+\infty} | f |^{\nu}_{\tau^{(m)}, \sigma} | h |_{\tau^{(m)}, w} \mathbf 1_{E_m} \sigma^{\nu}_{\tau^{(m)}}w_{\tau^{(m)}} \right] \label{as} \\
  & =&\mathbbm{E} \left[ \sum\limits_{m=0}^{+\infty} | f |^{\nu}_{\tau^{(m)}, \sigma} | h |_{\tau^{(m)}, w} \mathbf 1_{E_m} (w_{\tau^{(m)}}\sigma_{\tau^{(m)}}^{p-1}w_{\tau^{(m)}}^{-1})^{\frac{\nu}{p-1}}w_{\tau^{(m)}} \right]  \\
  &=&  \mathbbm{E} \left[ \sum\limits_{m=0}^{+\infty}
  A_{\tau^{(m)},p} (w)^{\frac{\nu}{p-1}} | f|^{\nu}_{\tau^{(m)}, \sigma} | h |_{\tau^{(m)}, w} \mathbf 1_{E_m} w^{1-\frac{\nu}{p-1}}_{\tau^{(m)}}  \right].
\end{eqnarray}

Using the conditional sparsity and H\"older’s inequality of the conditional expectation,  we have
 \begin{eqnarray*} \mathbf 1_{E_m} ^r&\lesssim &(\mathbf 1_{S_m})^{r}_{\tau^{(m)}}\mathbf 1_{E_m} \\
& =& (w^{\frac{1}{r}}w^{-\frac{1}{r}}\mathbf 1_{S_m})^{r}_{\tau^{(m)}}\mathbf 1_{E_m} \\
& \leq& (w\mathbf 1_{S_m})_{\tau^{(m)}}(w^{-\frac{r’}{r}}\mathbf 1_{S_m})^{\frac{r}{r’}}_{\tau^{(m)}}\mathbf 1_{E_m} \\
& \leq& (w\mathbf 1_{S_m})_{\tau^{(m)}}(w^{-\frac{r’}{r}})^{\frac{r}{r’}}_{\tau^{(m)}}\mathbf 1_{E_m} .
       \end{eqnarray*}

Thus
\begin{eqnarray*} w^{1-\frac{\nu}{p-1}}_{\tau^{(m)}} \mathbf 1_{E_m}
&=&(\mathbf 1^r_{E_m} w_{\tau^{(m)}})^{1-\frac{\nu}{p-1}}\mathbf 1_{E_m} \\
&\lesssim& A_{\tau^{(m)},r} (w)^{1-\frac{\nu}{p-1}} (w\mathbf 1_{S_m})^{1-\frac{\nu}{p-1}}_{\tau^{(m)}}\mathbf 1_{E_m}.
\end{eqnarray*}

It follows from $1-\frac{\nu}{p-1}=(-\frac{1}{p-1})\frac{\nu}{p}+(1-\frac{\nu}{p})$ that
\begin{eqnarray*}
 (w\mathbf 1_{S_m})^{1-\frac{\nu}{p-1}}_{\tau^{(m)}}\mathbf 1_{E_m}=\big((w\mathbf 1_{S_m})_{\tau^{(m)}}^{-\frac{1}{p-1}}\big) ^{\frac{\nu}{p}} (w\mathbf 1_{S_m})^{1-\frac{\nu}{p}}_{\tau^{(m)}}\mathbf 1_{E_m} .
\end{eqnarray*}

By H\"older’s inequality, we have
\begin{eqnarray*} &&\mathbbm{E} \left[ \sum\limits_{m=0}^{+\infty} (| f |\sigma)^{\nu}_{\tau^{(m)}} | h |_{\tau^{(m)}, w} \mathbf 1_{E_m} w_{\tau^{(m)}} \right] \\
  & \lesssim&    [w]^{\nu}_{(A_p)^{\frac{1}{p-1}} (A_r)^{\frac{1}{\nu}-\frac{1}{p-1}}} \left[ \mathbbm{E} \sum\limits_{m=0}^{+\infty}
   | f|^{p}_{\tau^{(m)}, \sigma} (w\mathbf 1_{S_m})_{\tau^{(m)}}^{-\frac{1}{p-1}} \right]^{\frac{\nu}{p}}\left[ \mathbbm{E} \sum\limits_{m=0}^{+\infty}
    h^{(\frac{p}{v})’}_{\tau^{(m)}, w}  (w\mathbf 1_{S_m})_{\tau^{(m)}} \right]^{{1-\frac{\nu}{p}}}.
      \end{eqnarray*}

Because
 \begin{eqnarray*} \mathbf 1_{E_m} &=&\mathbf 1_{E_m} ^p\lesssim(\mathbf 1_{S_m})^{p}_{\tau^{(m)}}\mathbf 1_{E_m} \\
& \leq& (w^{\frac{1}{p}}w^{-\frac{1}{p}}\mathbf 1_{S_m})^{p}_{\tau^{(m)}}\mathbf 1_{E_m} \\
& \leq& (w\mathbf 1_{S_m})_{\tau^{(m)}}(w^{-\frac{p^\prime}{p}}\mathbf 1_{S_m})^{\frac{p}{p^\prime}}_{\tau^{(m)}}\mathbf 1_{E_m} ,
       \end{eqnarray*}
we have $(w\mathbf 1_{S_m})^{-\frac{1}{p-1}}_{\tau^{(m)}}\mathbf 1_{E_m}\lesssim (\sigma\mathbf 1_{S_m})_{\tau^{(m)}}\mathbf 1_{E_m}.$

It follows that
\begin{eqnarray*} &&\mathbbm{E} \left[ \sum\limits_{m=0}^{+\infty} (| f |\sigma)^{\nu}_{\tau^{(m)}} | h |_{\tau^{(m)}, w} \mathbf 1_{E_m} w_{\tau^{(m)}} \right]\\
  & \lesssim& [w]^{\nu}_{(A_p)^{\frac{1}{p-1}} (A_r)^{\frac{1}{\nu}-\frac{1}{p-1}}}\left[ \mathbbm{E} \sum\limits_{m=0}^{+\infty}
   | f|^{p}_{\tau^{(m)},\sigma} (\sigma\mathbf 1_{S_m})_{\tau^{(m)}}\right]^{\frac{\nu}{p}}
   \left[ \mathbbm{E} \sum\limits_{m=0}^{+\infty}
   | h |^{(\frac{p}{v})’}_{\tau^{(m)}, w} (w\mathbf 1_{S_m})_{\tau^{(m)}} \right]^{{1-\frac{\nu}{p}}}\\
   &=& [w]^{\nu}_{(A_p)^{\frac{1}{p-1}} (A_r)^{\frac{1}{\nu}-\frac{1}{p-1}}} \left[ \mathbbm{E} \sum\limits_{m=0}^{+\infty}| f|^{p}_{\tau^{(m)}, \sigma} \sigma\mathbf 1_{S_m}\right]^{\frac{\nu}{p}}
   \left[ \mathbbm{E} \sum\limits_{m=0}^{+\infty}
   | h |^{(\frac{p}{v})’}_{\tau^{(m)}, w} w\mathbf 1_{S_m} \right]^{{1-\frac{\nu}{p}}}\\
   &\leqslant&  [w]^{\nu}_{(A_p)^{\frac{1}{p-1}} (A_r)^{\frac{1}{\nu}-\frac{1}{p-1}}}  \left[ \mathbbm{E}
  (M_{\sigma}  f)^{p} u\right]^{\frac{\nu}{p}}
   \left[ \mathbbm{E}
  (M_w  h )^{(\frac{p}{v})’}w\right]^{{1-\frac{\nu}{p}}}\\
   &\lesssim&  [w]^{\nu}_{(A_p)^{\frac{1}{p-1}} (A_r)^{\frac{1}{\nu}-\frac{1}{p-1}}} \left[ \mathbbm{E}(| f|)^{p} \sigma\right]^{\frac{\nu}{p}}
   \left[ \mathbbm{E}
 | h |^{(\frac{p}{v})’}w\right]^{{1-\frac{\nu}{p}}}\\
   &=&  [w]^{\nu}_{(A_p)^{\frac{1}{p-1}} (A_r)^{\frac{1}{\nu}-\frac{1}{p-1}}} \left[ \mathbbm{E}
  (| f|)^{p} \sigma\right]^{\frac{\nu}{p}}.
      \end{eqnarray*}

We conclude that
\begin{eqnarray*}
\| S_{\nu}( | f | \sigma) \|_{L^p (w)} &\lesssim& [w]_{(A_p)^{\frac{1}{p-1}} (A_r)^{\frac{1}{\nu}-\frac{1}{p-1}}} \left[ \mathbbm{E}
  (| f|)^{p} \sigma\right]^{\frac{1}{p}} . \end{eqnarray*}
\end{proof}

\section{Doob's Maximal Operator}\label{sec:mar_max}
We establish a simple sparse domination for Doob's maximal operator via a newly introduced method of locally filter-shifted stopping times. By simple, we mean that its domination is relative to that of uniformly bounded martingale transforms. 

\subsection{Simple Sparse Domination for  Doob's Maximal Operator}\label{subsec:mar_max}
\begin{lemma}\label{lem:max_essential}Let $(\Omega, \mathcal{F}, \mu, \{\mathcal{F}_n\}_{n \ge 0})$ be a filtered probability space. There is a constant \( C > 0 \) so that for all  \( f \in L^1 \), there is a stopping time $\tau$ on $(\Omega, \mathcal{F}, \mu, \{\mathcal{F}_n\}_{n \ge 0})$ such that \[Mf\leq 2|f|_0\mathbf 1_{\Omega\setminus E}+M^{\tau}f,\]
where $M^{\tau}f=(\sup\limits_{n\geq\tau}|f_n|) \mathbf 1_{\{\tau<\infty\}}.$
Furthermore, $\tau$ has the following properties
\begin{itemize}
    \item[(\ref{lem:max_essential}a)]  $\mu(\tau<\infty)\leq \frac{1}{2}\mu(\Omega).$
    \item[(\ref{lem:max_essential}b)]  For all $A \subset \Omega, A \in \mathcal{F}_0$ there holds
$\mathbf 1_{A}\leq 2\mathbb E(\mathbf 1_{E(A)}|\mathcal{F}_{0})\mathbf 1_{A},$
where $E(A)=A\backslash \{\tau<\infty\}=A\cap\{T=\infty\}.$
\end{itemize}
\end{lemma}

\begin{proof}[Proof of Lemma \ref{lem:max_essential}] The case $\|f\|_1=0$ is trivial with $\tau\equiv\infty$ and $M^{\tau}f=0.$
Let $\|f\|_1>0$. For a uniformly integral martingale
 $f = (f_n)_{n \ge 0}$ on $(\Omega, \mathcal{F}, \mu, \{\mathcal{F}_n\}_{n \ge 0})$ with $f_n=\mathbb E(f|\mathcal{F}_n),$ we
set
\begin{equation}
\label{eq:set_definition}
E = \{ Mf> C|f|_0\},
\end{equation}
where the constant \( C \) to be determined later. Let $\tau=\inf\{n:|f_n|>C|f|_0\}.$ Then $E=\{\tau<\infty\}.$ 
Denoting $\Omega_{0}=\{|f|_0=0\}$ and $\Omega_{1}=\{|f|_0\neq0\},$ we write 
\[\widetilde{\Omega}=\Omega_1,~\widetilde{\mathcal{F}} =\mathcal{F}|_{\Omega_1},~\widetilde{\mathcal{F}_n} =\mathcal{F}_n|_{\Omega_1},~\widetilde{\mu}= \frac{\mu|_{\Omega_{1}}}{\mu(\Omega_{1})}\] 
and \[\widetilde{f_n} =\frac{f_n|_{\Omega_1}}{|f|_0|_{\Omega_1}},~n\geq 0.\]
 It follows that $(\widetilde{f_n})_{n\geq 0}$ is a uniformly integral martingale on $(\widetilde{\Omega}, \widetilde{\mathcal{F}}, \widetilde{\mu}, 
\{\widetilde{\mathcal{F}_n}\}_{n \ge 0})$ and \[\lim_{n\rightarrow+\infty}\widetilde{f_n}=\frac{f|_{\Omega_1}}{|f|_0|_{\Omega_1}},~a.e.\] Here, $\frac{f|_{\Omega_1}}{|f|_0|_{\Omega_1}}$ is denoted by $\widetilde{f}.$
Recalling Remark \ref{re: uniform}, we have $\widetilde{f}=(\widetilde{f_n})_{n\geq 0}.$ 
Letting $\widetilde{\tau}=\inf\{n:|\widetilde{f}_n|>C\},$ we deduce that
$\tau|_{\Omega_1}=\widetilde{\tau}$ and 
\[
E= \left\{ M \widetilde{f} > C\right\}.
\]
Thus
\[
\widetilde{\mu}(E)
= \widetilde{\mu}(\left\{M \widetilde{f}>C\right\}) 
\leq\frac{ 1}{C}  \| \widetilde{f}\|_{L^1(\widetilde{\Omega}, \widetilde{\mathcal{F}}, \widetilde{\mu})}=\frac{ 1}{C},
\]
where we have used Lemma \ref{lem: weak}.  
Setting $C=2,$ we get $\widetilde{\mu}(E)\leq \frac{1}{2}$.

The case $\widetilde{\mu}(E)=0$ is trivial with $\tau\equiv\infty$ and $M^{\tau}f=0.$ 

We prove the case $0<\widetilde{\mu}(E)$ as follows.
Because $\widetilde{\mu}(E)\leq \frac{1}{2},$ we have
$\mu(E)\leq\frac{1}{2}\mu(\Omega_{1})\leq\frac{1}{2}\mu(\Omega).$
Then $\mu(\tau<\infty)\leq \frac{1}{2}\mu(\Omega),$ which is (\ref{lem:max_essential}a) and equivalent to
\[\mu(\Omega)\leq2\mu(\Omega\backslash\{\tau<\infty\})=2\mu(\Omega\cap\{\tau=\infty\}).\]

Let $A\in \mathcal{F}_0$.
Denoting $E(A)=A\backslash \{\tau<\infty\}=A\cap\{\tau=\infty\},$ we have
$$  \mu(A) \leqslant  2 \mu(E(A)),$$
by repeating the above proof and replacing $\Omega$ by $A.$
For any $A$ and $A_1\in \mathcal{F}_0,$ the same reasoning applied in the proof of Lemma \ref{lem:essential} shows that
\begin{eqnarray*}
\int_{A_1}\mathbf 1_{A}d\mu\leq2\int_{A_1}\mathbb E(\mathbf 1_{E(A)}|\mathcal{F}_0)d\mu.
\end{eqnarray*}
Since $A_1$ is arbitrary, we obtain $\mathbf 1_{A}\leq 2\mathbb E(\mathbf 1_{E(A)}|\mathcal{F}_0)\mathbf 1_{A}.$ This is (\ref{lem:max_essential}b).

For the above $\tau,$ we claim that
\[Mf\leq 2|f|_0\mathbf 1_{\Omega\setminus E}+M^{\tau}f,\]
where $M^{\tau}f=\sup\limits_{n\geq \tau}|f_n|1_{\{\tau<\infty\}}.$ 


Let $x\in \Omega\setminus E= \{\tau=\infty\},$ then $Mf(x)\leq2|f|_0(x).$ It follows that
\begin{equation}\label{eq:max_setminus}
|Mf|\mathbf 1_{\Omega\setminus E}\leq2|f|_0\mathbf 1_{\Omega\setminus E}.
\end{equation}

Let $x\in E.$ Then $x\in \{\tau<\infty\}.$ Suppose $\tau(x)=n_0.$ It follows that
\[(Mf)\mathbf 1_{\{\tau=n_0\}}=(\sup\limits_{n\geq n_0}|f_n|)\mathbf 1_{\{\tau=n_0\}},\]
which is a key observation.
Combining this with \eqref{eq:max_setminus}, we derive
\[Mf\leq 2|f|_0\mathbf 1_{\Omega\setminus E}+M^{\tau}f,\]
where $M^{\tau}f=\sup\limits_{n\geq \tau}|f_n|1_{\{\tau<\infty\}}.$ 
\end{proof}

\begin{lemma}\label{lem:max_induce}Let $l\in \mathbb {N}.$ For all  \( f \in L^1(\mu) \), there are stopping times $\tau_m$ on filtered probability spaces $(\Omega^{(m)}, \mathcal{F}^{(m)}, \mu^{(m)}, \{\mathcal{F}^{(m)}_n\}_{n \ge 0})$ such that \[Mf\leq 2\sum\limits_{m=1}^{l}|f|_{\tau^{(m-1)}}\mathbf 1_{\{\tau^{(m-1)}<\infty\}\backslash\{\tau^{(m)}<\infty\}}+M^{\tau^{(l)}}f,\]
where $M^{\tau^{(l)}}f=(\sup\limits_{n\geq\tau_l}|f_n|) \mathbf 1_{\{\tau^{(l)}<\infty\}}.$ 
Furthermore, $\tau^{(m)}$ has the following properties
\begin{equation}\label{eq:sparse}
\mathbf 1_{\{\tau^{(m-1)}<\infty\}}\leq 2\mathbb E(\mathbf 1_{\{\tau^{(m-1)}<\infty\}\backslash\{\tau^{(m)}<\infty\}}|\mathcal{F}_{\tau^{(m-1)}})\mathbf 1_{\{\tau^{(m-1)}<\infty\}},
\end{equation}
where $\tau^{(0)}$ is understood as $0.$
\end{lemma}

\begin{proof}[Proof of Lemma \ref{lem:max_induce}] 
We proceed by mathematical induction on \( l \).

\paragraph{Base Case (\( l = 1 \))}
When $l=1,$ we can prove the case by 
Lemma \ref{lem:max_essential}. Letting $$(\Omega^{(1)}, \mathcal{F}^{(1)}, \mu^{(1)}, \{\mathcal{F}^{(1)}_n\}_{n \ge 0}) =(\Omega, \mathcal{F}, \mu, \{\mathcal{F}_n\}_{n \ge 0})$$ and $\tau^{(1)}=\tau,$ we obtain that  
\[Mf\leq 2|f|_0\mathbf 1_{\Omega\setminus E}+M^{\tau^{(1)}}f,\]
with $M^{\tau^{(1)}}f=(\sup\limits_{n\geq\tau^{(1)}}|f_n|) \mathbf 1_{\{\tau^{(1)}<\infty\}}.$ This implies that
\[Mf\leq 2\sum\limits_{m=1}^{l}|f|_{\tau^{(m-1)}}\mathbf 1_{\{\tau^{(m-1)}<\infty\}\backslash\{\tau^{(m)}<\infty\}}+M^{\tau^{(l)}}f.\]

\paragraph{Inductive Hypothesis}
Assume inductively that stopping times $\tau^{(l)},$ $l=j$ have been constructed. 

\paragraph{Inductive Step (\( l = j + 1 \))}
For $l=j+1,$ we will define 
$\tau^{(j+1)}.$ Letting 
\begin{eqnarray*}
&&(\Omega^{(j+1)}, \mathcal{F}^{(j+1)}, \mu^{(j+1)}, \{\mathcal{F}^{(j+1)}_n\}_{n \ge 0}) \\
&=&(\{\tau^{(j)}<\infty\}, \mathcal{F}|_{\{\tau^{(j)}<\infty\}}, \frac{\mu|_{\{\tau^{(j)}<\infty\}}}{\mu\{\tau^{(j)}<\infty\}}, \{\mathcal{F}^{(j)}_{\tau^{(j)}+n}|_{\{\tau^{(j)}<\infty\}}\}_{n \ge 0}),  
\end{eqnarray*} 
and using Lemma \ref{lem:max_essential}, we obtain that  
\[|M^{\tau^{(j)}}f|\leq 2|f|_{\tau^{(j)}}\mathbf 1_{\{\tau^{(j)}<\infty\}\backslash\{\tau^{(j+1)}<\infty\}}+|M^{\tau^{(j+1)}}f|\]
with $M^{\tau^{(l+1)}}f=(\sup\limits_{n\geq\tau^{(l+1)}}|f_n|) \mathbf 1_{\{\tau^{(l+1)}<\infty\}}.$  Using the inductive hypothesis, we have
\[Mf\leq 2\sum\limits_{m=1}^{l+1}|f|_{\tau^{(m-1)}}\mathbf 1_{\{\tau^{(m-1)}<\infty\}\backslash\{\tau^{(m)}<\infty\}}+M^{\tau^{(l+1)}}f.\]

\paragraph{Conclusion}
By the principle of mathematical induction, the formula holds for all positive integers \( l \).
\end{proof}

\begin{proof}[Proof of Theorem \ref{thm:max_sparse}] 
Following Lemma \ref{lem:max_induce}, we have $\lim\limits_{m\rightarrow\infty}\mu(\{\tau^{(m)}<\infty\})=0.$ Then
\[Mf\leq 2\sum\limits_{m=1}^{+\infty}|f|_{\tau^{(m-1)}}\mathbf 1_{\{\tau^{(m-1)}<\infty\}\backslash\{\tau^{(m)}<\infty\}}.\]This is \eqref{eq:max_main} with $\mathbb{S}|f|=\sum\limits_{m=0}^{+\infty}|f|_{\tau^{(m)}}\mathbf 1_{\{\tau^{(m)}<\infty\}\backslash\{\tau^{(m+1)}<\infty\}}=\sum\limits_{m=0}^{+\infty}|f|_{\tau^{(m)}}\mathbf 1_{S_m}.$
\end{proof}

\subsection{Weighted Inequalities for Doob's Maximal Operator}A major advantage of this simple sparse domination is its ability to deliver $A_p$ weighted estimates for simple sparse operators without recourse to dual theory, while simultaneously yielding the optimal estimates for Doob's maximal operator---one of the key contributions of this work.

\begin{proof}[Proof of  Theorem \ref{thm_Ap}]
($\Rightarrow$) For $p > 1$, fix $n$ and let $B \in \mathcal{F}_n$. Define $f = \sigma \chi_B$. Then $f\in L^1.$ It follows that
\begin{align*}
\int_B \mathbb{E} \left(  \sigma \mid \mathcal{F}_n \right)^p \mathbb{E}(v \mid \mathcal{F}_n) \, d\mu 
&=\int_{\Omega} \mathbb{E} \left(  \sigma\mathbf 1_B \mid \mathcal{F}_n \right)^p \mathbb{E}(v \mid \mathcal{F}_n) \, d\mu \\
&\leq \int_{\Omega} \left(  Mf \right)^p w \, d\mu \\
&\leq \|M\|^p\int_{\Omega}  f^pw d\mu \\
&= \|M\|^p\int_{B}  \sigma d\mu \\
&= \|M\|^p \int_B \mathbb{E} \left(  \sigma \mid \mathcal{F}_n \right) \, d\mu.
\end{align*}
Since $\mathbb{E}(\sigma\mid \mathcal{F}_n) < \infty$ almost everywhere, this establishes $[w]_{A_p}\leq\|M\|^p.$

($\Leftarrow$) Theorem \ref{thm:max_sparse} shows that
\begin{equation}\label{doob}
\int_{\Omega}{M}
(f\sigma)^{p}wd\mu
\leq \sum\limits_{m=0}^{+\infty}\int_{S_m}|f\sigma|_{\tau^{(m)}}^pwd\mu.
\end{equation}

We now estimate $\int_{S_m}|f\sigma|_{\tau^{(m)}}^pwd\mu$ as follows:

\begin{eqnarray*}
\int_{S_m}|f\sigma|_{\tau^{(m)}}^pwd\mu&\leq&\int_{E_m}|f\sigma|_{\tau^{(m)}}^pwd\mu\\
&=&\int_{E_m}|f\sigma|_{\tau^{(m)}}^pw_{\tau^{(m)}}^{p^\prime}w_{\tau^{(m)}}^{1-p^\prime}
\sigma_{\tau^{(m)}}^{p}\sigma_{\tau^{(m)}}^{-p}d\mu\\
&=&\int_{E_m}|f\sigma|_{\tau^{(m)}}^pw_{\tau^{(m)}}^{p^\prime}
\sigma_{\tau^{(m)}}^{p}w_{\tau^{(m)}}^{1-p^\prime}\sigma_{\tau^{(m)}}^{-p}d\mu.
\end{eqnarray*}
In the view of the definition of $A_p$ and the construction of $\tau^{(m)},$ we have
\begin{eqnarray*}
\int_{S_m}|f\sigma|_{\tau^{(m)}}^pwd\mu&\leq&[w]^{p^\prime}_{A_{p}}\int_{E_m}|f\sigma|_{\tau^{(m)}}^pw_{\tau^{(m)}}^{1-p'}
\sigma_{\tau^{(m)}}^{-p}d\mu\\
&\lesssim&[w]^{p^\prime}_{A_{p}}\int_{E_m}|f\sigma|_{\tau^{(m)}}^pw_{\tau^{(m)}}^{1-p^\prime}
    \sigma_{\tau^{(m)}}^{-p}
    (\mathbf 1_{S_m})_{\tau^{(m)}}^{p(p'-1)}d\mu\\
&=&[w]^{p^\prime}_{A_{p}}\int_{E_m}|f\sigma|_{\tau^{(m)}}^pw_{\tau^{(m)}}^{1-p^\prime}
    \sigma_{\tau^{(m)}}^{-p}\mathbb (\mathbf 1_{S_m}w^{\frac{1}{p}}\sigma^{\frac{1}{p'}})_{\tau^{(m)}}^{p(p'-1)}d\mu.
\end{eqnarray*}
Noting that the conditional expectation satisfies H\"{o}lder's inequality,  we have
\begin{eqnarray*} \int_{S_m}|f\sigma|_{\tau^{(m)}}^pwd\mu
&\lesssim&[w]^{p^\prime}_{A_{p}}\int_{E_m}|f\sigma|_{\tau^{(m)}}^pw_{\tau^{(m)}}^{1-p^\prime}
\sigma_{\tau^{(m)}}^{-p}\\
&&\times(w \mathbf 1_{S_m})_{\tau^{(m)}}^{p^\prime-1}
(\sigma\mathbf 1_{S_m})_{\tau^{(m)}}d\mu\\
&\leq&[w]^{p^\prime}_{A_{p}}\int_{E_m}|f\sigma|_{\tau^{(m)}}^p
\sigma_{\tau^{(m)}}^{-p}
(\sigma\mathbf 1_{S_m})_{\tau^{(m)}}
d\mu.\end{eqnarray*}
Then we obtain that
\begin{eqnarray*}
 \int_{S_m}|f\sigma|_{\tau^{(m)}}^pwd\mu
&\lesssim&[w]^{p^\prime}_{A_{p}}\int_{E_m}|f\sigma|_{\tau^{(m)}}^p\sigma_{\tau^{(m)}}^{-p} (\mathbf 1_{S_m}\sigma)_{\tau^{(m)}}d\mu\\
&=&[w]^{p^\prime}_{A_{p}}\int_{E_m}|f|_{\tau^{(m)},\sigma}^p (\mathbf 1_{S_m}\sigma)_{\tau^{(m)}}d\mu,
\end{eqnarray*}
where we have used $ |f\sigma|_{\tau^{(m)}}=|f|_{_{\tau^{(m)},\sigma}}\sigma_{\tau^{(m)}}.$
Then
\begin{eqnarray*}
 \int_{S_m}|f\sigma|_{\tau^{(m)}}^pwd\mu
&\lesssim&[w]^{p^\prime}_{A_{p}}\int_{E_m}|f|_{\tau^{(m)},\sigma}^p(\mathbf 1_{S_m}\sigma)_{\tau^{(m)}}d\mu\\
&=&[w]^{p^\prime}_{A_{p}}\int_{E_m}|f|_{\tau^{(m)},\sigma}^p \mathbf 1_{S_m}\sigma d\mu\\
&\leq&[w]^{p^\prime}_{A_{p}}\int_{E_m}M^{\sigma}(f)^p \mathbf 1_{S_m}\sigma d\mu\\
&=&[w]^{p^\prime}_{A_{p}}\int_{S_m}M^{\sigma}(f)^p \sigma d\mu.
\end{eqnarray*}
It follows from \eqref{doob} and the boundedness of Doob's maximal operator $M^{\sigma}$ that
\begin{eqnarray*}
\int_{\Omega}{M}(f\sigma)^{p}wd\mu
&\lesssim&[w]^{p^\prime}_{A_{p}}\sum\limits_{m=0}^{+\infty}\int_{S_m}M^{\sigma}(f)^p \sigma d\mu\\
&=&[w]^{p^\prime}_{A_{p}}\int_{\Omega}M^{\sigma}(f)^p \sigma d\mu\\
&\lesssim&[w]^{p^\prime}_{A_{p}}\int_{\Omega}|f|^p\sigma d\mu,
\end{eqnarray*}
which implies 
\begin{eqnarray*}
\big(\int_{\Omega}{M}(f\sigma)^{p}wd\mu\big)^{\frac{1}{p}}
&\lesssim&[w]^{\frac{p^\prime}{p}}_{A_{p}}\big(\int_{\Omega}|f|^p\sigma d\mu\big)^{\frac{1}{p}}.
\end{eqnarray*}
Then we have $\|M\|\lesssim[w]^{\frac{1}{p-1}}_{A_p}.$
\end{proof}

\subsection{Further Two-Weight Applications of Sparse Domination}
We now focus on the application of sparse theory to quantitative two-weight estimates in Theorem \ref{thm-es}.
The proofs of these estimates rely essentially on sparse domination techniques. By establishing a pointwise sparse control of  Doob's maximal operator, we reduce the problem to estimating simple sparse operators. This approach provides a unified and transparent framework for handling all three cases in the theorem.
Before proceeding, we recall the following Lemma \ref{sa_test}, which was proved in \cite{MR2834220}.
It gives Sawyer's characterization \cite{MR676801} of the two-weight inequality for  Doob's maximal operator in the present setting.

\begin{lemma}\label{sa_test} \cite[Theorem 3.1]{MR2834220} Let $1<p<\infty$
and let $(u,w)$ be a pair of weights. Denote $\sigma=w^{-\frac{1}{p-1}}\in L^1.$ Then the following are
equivalent:
\begin{enumerate}
 \item \label{testno1} There exists a positive constant $C_1$ such that
\begin{equation}
\label{test1}\Big(\int_{\{\tau<\infty\}}\big(M(\sigma
\mathbf 1_{\{\tau<\infty\}})\big)^pud\mu\Big)^{\frac{1}{p}}
 \leq
C_1(\int_{\{\tau<\infty\}}\sigma d\mu\Big)^{\frac{1}{p}}, ~~ \forall \tau;
\end{equation}
\item\label{testno2}There exists a positive constant $C_2$ such that
\begin{equation}
\label{test2} \big (\mathbbm{E}({^*\sigma_n^p}u|\mathcal {F}_n)\big)^{\frac{1}{p}}
 \leq
C_2 (\sigma_n)^{\frac{1}{p}},  ~~\forall n\geq0,
\end{equation}
where ${^*\sigma_n}=\sup_{m\geq n}\sigma_m.$
\item\label{testno3}There exists a positive constant $C_3$ such that
\begin{equation}
\label{test3}\Big(\int_{\Omega}(Mf)^pud\mu\Big)^{\frac{1}{p}}\leq C_3\Big(\int_\Omega|f|^pw d\mu\Big)^{\frac{1}{p}},
~~ \forall f=(f_n)\in L^p(w).
\end{equation}
\end{enumerate}
Moreover, we denote the smallest positive constants $C_1,$ $C_2,$ and $C_3$ in \eqref{test1}, \eqref{test2}, and \eqref{test3}
by $[u,w]_{S_p}$, $[u,w]_{S^*_p}$, and $\|M\|_{L^p(w)\rightarrow L^p(u)}$, respectively. Then it follows that $[u,w]_{S_p}\approx[u,w]_{S^*_p}\approx\|M\|_{L^p(w)\rightarrow L^p(u)}.$
\end{lemma}

\begin{remark}\label{rem_test} For $B\in \mathcal {F}_n,$ we have
\begin{align}\label{eq_test}
\int_B\mathbbm{E}\big({^*\sigma}_n^pu|\mathcal
{F}_n\big)d\mu&=\int_B{^*\sigma_n^pud\mu}=\int_B{^*\sigma_n^p \mathbf 1_{B}ud\mu}=\int_B{{^*M}_n(\sigma\mathbf 1_{B})^p ud\mu}.
\end{align}
Thus, Item \eqref{testno2} of Lemma \ref{sa_test} is equivalent to that
there exists a positive constant $C$ such that
\begin{equation}\label{test_sp}
\int_B{{^*M}_n(\sigma\mathbf 1_{B})^p ud\mu}\leq C^p\int_B\sigma
d\mu,~\forall~B\in \mathcal{F}_n.
\end{equation}
This inequality characterizes the condition \( S_p^* \) as given in \eqref{bi-SP} of Definition \ref{definition Sp}. We may therefore unambiguously denote the smallest such constant \( C \) satisfying \eqref{test2} by \( [u, \sigma]_{S_p^*} \).
\end{remark}

\begin{proof}[Proof of Theorem \ref{thm-es}]
In view of Theorem \ref{sa_test} and Remark \ref{rem_test}, it suffices to check \eqref{test_sp}.
Employing Theorem \ref{thm:max_sparse} for $ \sigma \mathbf 1_B,$ we obtain a sparse sequence of stopping times $\mathcal{T}=\{ \tau^{(m)}\}_{m\geqslant 0}$ such that
\begin{equation*}
\mathbf 1_{B} \cdot {^*M}_n (\sigma \mathbf 1_{B}) \lesssim \sum\limits_{m=0}^{+\infty} (\sigma \mathbf 1_B)_{\tau^{(m)}} \mathbf 1_{E_m\setminus E_{m+1}}=\sum\limits_{m=0}^{+\infty} (\sigma \mathbf 1_B)_{\tau^{(m)}} \mathbf 1_{S_m}.
\end{equation*} 
Because $B$ is fixed,
$\mathbf 1_B$ is suppressed in the notation. The sets $ S_m$ are pairwise disjoint, hence,
\begin{equation}\label{eq:doobmax-domi}
\int _{B} {^*M}_n (\sigma \mathbf 1_{B}) ^{p} u d\mu
 \lesssim
\sum\limits_{m=0}^{+\infty} \int_{E_m}\sigma_{\tau^{(m)}} ^{p} ud\mu .
\end{equation}

{Proof of \eqref{Bound1}.} It follows from the definition of $B_p$ that
\begin{eqnarray*} \int_{E_m}\sigma_{\tau^{(m)}} ^{p} u d\mu
&=& \int_{E_m}\sigma_{\tau^{(m)}} ^{p} u_{\tau^{(m)}}  d\mu\\
&\leq&[u,w]_{B_p}\int_{E_m}\exp(\mathbb (\log\sigma)_{\tau^{(m)}})d\mu.\end{eqnarray*}
Note that
\begin{eqnarray*}
\int_{E_m}\exp(\mathbb (\log\sigma)_{\tau^{(m)}})d\mu
&=&\int_{E_m}\exp(\mathbb (\log\sigma)_{\tau^{(m)}})\mathbf 1_{E_m}d\mu.\end{eqnarray*}
It follows that
\begin{eqnarray*}
 \int_{E_m}\sigma_{\tau^{(m)}} ^{p} u d\mu
&\lesssim&[u,w]_{B_p}\int_{E_m}\exp(\mathbb (\log\sigma)_{\tau^{(m)}})(\mathbf 1_{S_m})_{\tau^{(m)}}d\mu\\
&=&[u,w]_{B_p}\int_{E_m}\exp(\mathbb (\log\sigma)_{\tau^{(m)}})1_{S_m}d\mu\\
&=&[u,w]_{B_p}\int_{S_m}\exp(\mathbb (\log\sigma)_{\tau^{(m)}})d\mu.\end{eqnarray*}
Using Jensen's inequality for conditional expectation, for any $q>1,$ we have
$$\exp( (\log\sigma)_{\tau^{(m)}})\mathbf 1_{S_m}
\leq \big((\sigma^{\frac{1}{q}})_{\tau^{(m)}}\big)^q\mathbf 1_{S_m}
\leq  {^*M}_n(\sigma^{\frac{1}{q}})^q\mathbf 1_{S_m}.$$
Then
\begin{eqnarray*}\int _{B} {^*M}_n (\sigma \mathbf 1_{B}) ^{p} w d\mu
&\lesssim&[u,w]_{B_p}\sum\limits_{m=0}^{+\infty} \int_{S_m} {^*M}_n(\sigma^{\frac{1}{q}})^qd\mu.\\
&\leq&[u,w]_{B_p}\int_{B} {^*M}_n(\sigma^{\frac{1}{q}})^qd\mu.\end{eqnarray*}
Combining it with the boundedness of Doob's maximal operator, we deduce that
\begin{eqnarray*}
\int _{B} {^*M}_n (\sigma \mathbf 1_{B}) ^{p} w d\mu
&\lesssim&[u,w]_{B_p}(q')^{q}\int_{B}\sigma d\mu.
\end{eqnarray*}
Letting $q\rightarrow+\infty,$ we obtain $(q')^{q}\rightarrow e.$ Thus
\begin{eqnarray*}\int_{B}{^*M_n}(\sigma)^{p}u d\mu\lesssim[u,w]_{B_p}\sigma(B).
\end{eqnarray*}

{Proof of \eqref{Bound2}.} It follows from the definition of $A_p$ that
\begin{eqnarray*}\int_{E_m}\sigma_{\tau^{(m)}} ^{p} u d\mu
&=& \int_{E_m}\sigma_{\tau^{(m)}} ^{p} u_{\tau^{(m)}}  d\mu\\
&\lesssim&[u,w]_{A_p}\int_{E_m}\sigma_{\tau^{(m)}}d\mu.\end{eqnarray*}
Note that $\int_{E_m}\sigma_{\tau^{(m)}}d\mu=\int_{E_m}\sigma_{\tau^{(m)}}\mathbf 1_{E_m}d\mu.$
It follows that
\begin{eqnarray*}
\int_{E_m}\sigma_{\tau^{(m)}}\mathbf 1_{E_m}d\mu
&\lesssim&\int_{E_m}\sigma_{\tau^{(m)}}(\mathbf 1_{S_m})_{\tau^{(m)}}d\mu\\
&=&\int_{S_m}\mathbb\sigma_{\tau^{(m)}}d\mu\\
&\leq&\int_{S_m}{^*M_{n}}(\sigma)d\mu.
\end{eqnarray*}
Then
\begin{eqnarray*}\int _{B} {^*M}_n (\sigma \mathbf 1_{B}) ^{p} w d\mu
&\lesssim&[u,w]_{A_p}\sum\limits_{m=0}^{+\infty} \int_{S_m}{^*M_{n}}(\sigma)d\mu.\\
&\leq&[u,w]_{A_p}\int_{B}{^*M_{n}}(\sigma)d\mu\\
&=&[u,w]_{A_p}\int_{B}(\frac{{^*M_n\sigma}}{\sigma_n})_n\cdot\sigma_nd\mu.
\end{eqnarray*}
Because of $\sigma\in A^*_{\infty},$ we have
\begin{eqnarray*}\int _{B} {^*M}_n (\sigma \mathbf 1_{B}) ^{p} w d\mu\lesssim[u,w]_{A_p}[\sigma]_{A^*_{\infty}}\sigma(B).\end{eqnarray*}

{Proof of \eqref{Bound3}.} 
For $a\in \mathbb Z,$ define $$\mathcal{T}^a=\{\tau^{(m)}\in\mathcal{T}:2^{a-1}
<\mathop{\hbox{esssup}}\limits_{E_m}(\sigma_{\tau^{(m)}} ^{p-1} w_{\tau^{(m)}}\mathbf 1_{E_m})\leq2^a\}.$$
It follows from H\"{o}lder's inequality that $1= \mathbb E_n(w^{\frac{1}{p}}w^{-\frac{1}{p}})^p\leq \mathbb E_n(w)\mathbb E_n(\sigma)^{p-1}\leq [w]_{A_p}, $ for any $n\in \mathbb N.$ Setting $K=[\log_2[w]_{A_p}]+1,$ we have
$$\mathcal{T}=\bigcup\limits_{a=0}^{K}\mathcal{T}^a.$$
Let $u=w$ in \eqref{eq:doobmax-domi}.  Then we obtain that
 \begin{eqnarray*}\int _{B} {^*M}_n (\sigma \mathbf 1_{B}) ^{p} w d\mu
&\lesssim& 
\sum\limits_{m=0}^{+\infty} \int_{E_m}\sigma_{\tau^{(m)}} ^{p} w d\mu\\
&=&\sum\limits_{m=0}^{+\infty} \int_{E_m}\sigma_{\tau^{(m)}} ^{p} w_{\tau^{(m)}}  d\mu.\end{eqnarray*}
Note that
\begin{eqnarray*}\Big(\sigma_{\tau^{(m)}} ^{p} w_{\tau^{(m)}}\Big)\mathbf 1_{E_m}
\leq\sigma_{\tau^{(m)}}\mathbf 1_{E_m}\mathop{\hbox{esssup}}\limits_{E_m}(\sigma_{\tau^{(m)}} ^{p-1} w_{\tau^{(m)}}\mathbf 1_{E_m})
.\end{eqnarray*}
It follows that
\begin{eqnarray*}
\sum\limits_{m=0}^{+\infty} \int_{E_m}\sigma_{\tau^{(m)}} ^{p} w_{\tau^{(m)}}  d\mu
&\leq&\sum\limits_{a=0}^{K}2^a\sum\limits_{\tau^{(m)}\in \mathcal{T}^a}\int_{E_m}
\mathbb \sigma_{\tau^{(m)}}d\mu.\end{eqnarray*}
Therefore,
\begin{eqnarray*}\sum\limits_{\tau^{(m)}\in \mathcal{T}^a}\int_{E_m} \sigma_{\tau^{(m)}} d\mu
&=&\sum\limits_{\tau^{(m)}\in \mathcal{T}^a}\int_{E_m} \sigma_{\tau^{(m)}} \mathbf 1_{E_m}d\mu \\
&\lesssim&\sum\limits_{\tau^{(m)}\in \mathcal{T}^a}\int_{E_m} \sigma_{\tau^{(m)}} (\mathbf 1_{S_m})_{\tau^{(m)}}d\mu \\
&=&\sum\limits_{\tau^{(m)}\in \mathcal{T}^a}\int_{E_m} \sigma_{\tau^{(m)}}\mathbf 1_{S_m}d\mu.
\end{eqnarray*}
Denoting $m(a)=\inf\{m:\tau^{(m)}\in \mathcal T _{a}\},$  we have $ \sigma_{\tau^{(m)}}\cdot\mathbf 1_{S_m}\leq{^*M_{\tau^{(m(a))}}\sigma}\cdot\mathbf 1_{S_m},$ where $\tau^{(m)}\in \mathcal T _{a}.$
Thus
\begin{eqnarray*}\sum\limits_{\tau^{(m)}\in \mathcal{T}^a}\int_{E_m} \sigma_{\tau^{(m)}} d\mu
&\lesssim&\sum\limits_{\tau^{(m)}\in \mathcal{T}^a}\int_{S_m} {^*M_{\tau^{(m(a))}}\sigma}d\mu\\
&\leq&\int_{E_{m(a)}} {^*M_{\tau^{(m(a))}}\sigma}d\mu.\end{eqnarray*}

Furthermore, we derive the bound
\begin{eqnarray*}\int _{B} {^*M}_n (\sigma \mathbf 1_{B}) ^{p} w d\mu
&\lesssim& \sum\limits_{a=0}^{K}2^a\int_{E_{m(a)}} {^*M_{\tau^{(m(a))}}\sigma}d\mu\\
&\lesssim& \sum\limits_{a=0}^{K}
\int_{E_{m(a)}}\sigma_{\tau^{(m(a))}} ^{p-1} w_{\tau^{(m(a))}}\cdot(\frac{{^*M_{\tau^{(m(a))}}\sigma}}{\sigma_{\tau^{(m(a))}}})_{\tau^{(m(a))}}\cdot\sigma_{\tau^{(m(a))}}d\mu.\end{eqnarray*}
By \eqref{mixcon1} the definition of $(A_{p'})^{\frac{1}{p'}}(A^*_{\infty})^{\frac{1}{p}},$ we have
\begin{eqnarray*}\int _{B} {^*M}_n (\sigma \mathbf 1_{B}) ^{p} w d\mu
&\lesssim& [\sigma]^p_{(A_{p'})^{\frac{1}{p'}}(A^*_{\infty})^{\frac{1}{p}}}\sum\limits_{a=0}^{K}\int_{E_{m(a)}}\sigma_{\tau^{(m(a))}}d\mu\\
&=& [\sigma]^p_{(A_{p'})^{\frac{1}{p'}}(A^*_{\infty})^{\frac{1}{p}}}\sum\limits_{a=0}^{K}\int_{E_{m(a)}}\sigma d\mu\\
&\leq& [\sigma]^p_{(A_{p'})^{\frac{1}{p'}}(A^*_{\infty})^{\frac{1}{p}}}\sum\limits_{a=0}^{K}\int_{B}\sigma d\mu\\
&=&(K+1) [\sigma]^p_{(A_{p'})^{\frac{1}{p'}}(A^*_{\infty})^{\frac{1}{p}}}\sigma(B).\end{eqnarray*}
Combining the previous estimates, we conclude that
\begin{eqnarray*}
\int _{B} {^*M}_n (\sigma \mathbf 1_{B}) ^{p} w d\mu
&\lesssim& [\sigma]^p_{(A_{p'})^{\frac{1}{p'}}(A^*_{\infty})^{\frac{1}{p}}}(3+\log_2[w]_{A_p})\sigma(B)\\
&\lesssim& [\sigma]^p_{(A_{p'})^{\frac{1}{p'}}(A^*_{\infty})^{\frac{1}{p}}}(1+\log_2[w]_{A_p})\sigma(B),
\end{eqnarray*}
which completes the proof of  Theorem \ref{thm-es}.
\end{proof}

\bibliographystyle{alpha,amsplain}	

\end{document}